\documentclass{amsart}
\usepackage{eurosym}
\usepackage{amssymb}
\usepackage{amsmath}
\usepackage{amsfonts}
\usepackage{graphicx,color}

\newtheorem{theorem}{Theorem}
\theoremstyle{plain}
\newtheorem{acknowledgement}{Acknowledgements}

\newtheorem{lemma}{Lemma}

\newtheorem{proposition}{Proposition}
\newtheorem{remark}{Remark}

\numberwithin{equation}{section}

\input{tcilatex}

\begin{document}
\title[Quasilinear singular elliptic systems]{Positive solutions for a class
of quasilinear singular elliptic systems}
\subjclass[2010]{35J75; 35J48; 35J92}
\keywords{Singular system; $p$-Laplacian; Leray-Schauder degree; regularity}

\begin{abstract}
In this paper we establish the existence of two positive solutions for a
class of quasilinear singular elliptic systems. The main tools are sub and
supersolution method and Leray-Schauder Topological degree.
\end{abstract}

\author{Claudianor O. Alves}
\address{Claudianor O. Alves\\
Universidade Federal de Campina Grande, Unidade Acad\^{e}mica de Matem\'{a}%
tica, CEP:58429-900, Campina Grande - PB, Brazil.}
\email{coalves@mat.ufcg.edu.br}
\thanks{C.O. Alves was partially supported by CNPq/Brazil 304036/2013-7 and
INCT-MAT}
\author{Abdelkrim Moussaoui}
\address{Abdelkrim Moussaoui\\
Biology Department, A. Mira Bejaia University, Targa Ouzemour, 06000 Bejaia,
Algeria.}
\email{abdelkrim.moussaoui@univ-bejaia.dz}
\thanks{A. Moussaoui was supported by CNPq/Brazil 402792/2015-7.}
\maketitle

\section{Introduction}

\label{S1}

We consider the following system of quasilinear elliptic equations:%
\begin{equation}
\left\{ 
\begin{array}{ll}
-\Delta _{p}u=u^{\alpha _{1}}v^{\beta _{1}} & \text{in }\Omega , \\ 
-\Delta _{q}v=u^{\alpha _{2}}v^{\beta _{2}} & \text{in }\Omega , \\ 
u,v>0 & \text{in }\Omega , \\ 
u,v=0 & \text{on }\partial \Omega ,%
\end{array}%
\right.  \tag{$P$}  \label{p}
\end{equation}%
where $\Omega $ is a bounded domain in $%
\mathbb{R}
^{N}$ $\left( N\geq 2\right) $ with $C^{1,\alpha }$ boundary $\partial
\Omega $, $\alpha \in (0,1)$, $\Delta _{p}$ and $\Delta _{q},$ $1<p,q<N,$
are the $p$-Laplacian and $q$-Laplacian operators, respectively, that is, $%
\Delta _{p}u=div\left( \left\vert \nabla u\right\vert ^{p-2}\nabla u\right) $
and $\Delta _{q}v=div\left( \left\vert \nabla v\right\vert ^{q-2}\nabla
v\right) .$ We consider the system (\ref{p}) in a singular case by assuming
that%
\begin{equation}
\left\{ 
\begin{array}{c}
-1<\alpha _{1}<0<\beta _{1}<\min \{p-1,\frac{q^{\ast }}{p^{\ast }}%
(p-1-\alpha _{1})\} \\ 
-1<\beta _{2}<0<\alpha _{2}<\min \{q-1,\frac{p^{\ast }}{q^{\ast }}(q-1-\beta
_{2})\}.%
\end{array}%
\right.  \label{h1}
\end{equation}%
In this case, system (\ref{p}) is cooperative, that is, for $u$ (resp. $v$)
fixed the right term in the first (resp. second) equation of (\ref{p}) is
increasing in $v$ (resp. $u$).

The study of singular elliptic problems is greatly justified because they
arise in several physical situations such as fluid mechanics pseudoplastics
flow, chemical heterogeneous catalysts, non-Newtonian fluids, biological
pattern formation and so on. In Fulks \& Maybee \cite{FM}, the reader can
find a very nice physical illustration of a practical problem which leads to
singular problem.

With respect to singular system it is worth to cite, among others, the
important Gierer-Meinhardt system which is the stationary counterpart of a
parabolic system proposed by Gierer-Meinhardt (see \cite{GM1, D}) which
occurs in the study of morphogenesis on experiments on hydra, an animal of a
few millimeters in length.

Besides the importance of the physical application above mentioned, we would
like to mention that from a mathematical point of view the singular problems
are also interesting because to solve some of them are necessary nontrivial
mathematical techniques, which involve Topological degree, Bifurcation
theory, Fixed point theorems, sub and supersolution Method, Pseudomonotone
Operator theory and Variational Methods. Here, it is impossible to cite all
papers in the literature which use the above techniques, however the reader
can find the applications of the above mentioned methods in Alves \&
Moussaoui \cite{CM}, Hai \cite{Hai}, Ghergu \& Radulescu \cite{GR},
Giacomoni, Hernandez \& Moussaoui \cite{GHM}, Giacomoni, Hernandez \& Sauvy 
\cite{GHS}, Hernandez, Mancebo \& Vega, \cite{HMV}, Khodja \& Moussaoui \cite%
{KM}, Zhang \cite{Z1}, Zhang \& Yu \cite{ZY}, Diaz, Morel \& Oswald \cite%
{DMO}, Alves, Corr\^{e}a \& Gon\c{c}alves \cite{ACG}, Crandall \& Rabinowitz 
\cite{CR}, Taliaferro \cite{T}, Lunning \& Perry \cite{LP}, Motreanu \&
Moussaoui \cite{MM2, MM3, MM}, Moussaoui, Khodja \& Tas \cite{MKT}, Agarwall
and O'Regan \cite{AO}, Stuart \cite{ST} and their references.

After a review bibliography, we did not find any paper where the existence
of multiple solutions have been considered for a singular system. Motivated
by this fact, we prove in the present paper the existence of at least two
positive solutions for system $(P)$. Our main result has the following
statement:

\begin{theorem}
\label{T2} Under assumption (\ref{h1}) problem (\ref{p}) possesses at least
two (positive) solutions in $C^{1,\gamma }(\overline{\Omega })\times
C^{1,\gamma }(\overline{\Omega }),$ for certain $\gamma \in (0,1)$.
\end{theorem}

In the proof of the above theorem, we will use sub and supersolution method
combined with Leray-Schauder Topological degree. However, before proving
that theorem it was necessary to get some informations about the regularity
of the solutions. To this end, the below result was crucial in our approach.

\begin{theorem}
\label{T1} Assume (\ref{h1}) holds. Then, system (\ref{p}) has a positive
solution $\left( u,v\right) $ in $C^{1,\gamma }(\overline{\Omega })\times
C^{1,\gamma }(\overline{\Omega })$ for some $\gamma \in (0,1)$. Moreover,
there exist a sub-supersolution $\left( \underline{u},\underline{v}\right) ,(%
\overline{u},\overline{v})\in C^{1}(\overline{\Omega })\times C^{1}(%
\overline{\Omega })$ for (\ref{p}) such that%
\begin{equation}
\underline{u}(x)\leq u(x)\leq \overline{u}(x)\text{ and }\underline{v}%
(x)\leq v(x)\leq \overline{v}(x)\text{ for all }x\in \overline{\Omega }.
\label{c}
\end{equation}
\end{theorem}

In the present paper, a solution of (\ref{p}) is understood in the weak
sense, that is, a pair $(u,v)\in W_{0}^{1,p}(\Omega )\times
W_{0}^{1,q}(\Omega )$, with $u,v$ positive a.e. in $\Omega ,$ satisfying%
\begin{equation}
\left\{ 
\begin{array}{cc}
\int_{\Omega }|\nabla u|^{p-2}\nabla u\nabla \varphi \ dx & =\int_{\Omega
}u^{\alpha _{1}}v^{\beta _{1}}\varphi \ dx, \\ 
\int_{\Omega }|\nabla v|^{q-2}\nabla v\nabla \psi \ dx & =\int_{\Omega
}u^{\alpha _{2}}v^{\beta _{2}}\psi \ dx,%
\end{array}%
\right.  \label{7}
\end{equation}%
for all $(\varphi ,\psi )\in W_{0}^{1,p}(\Omega )\times W_{0}^{1,q}(\Omega )$%
.

The proof of Theorem \ref{T1} is done in Section \ref{S2}. The main technica%
{\small l }difficulty consists in the presence{\small \ }of singular terms
in system (\ref{p}) under condition (\ref{h1}). Our approach is based on the
sub-supersolution method in its version for systems \cite[section 5.5]{CLM}.
However, this method cannot be directly implemented due to the presence of
singular terms in system (\ref{p}). Applying the sub-supersolution method in
conjunction with the regularity result in \cite{Hai} under hypothesis (\ref%
{h1}), we prove the existence of a (positive) solution $(u,v)\in C^{1,\gamma
}(\overline{\Omega })\times C^{1,\gamma }(\overline{\Omega }),$ for certain $%
\gamma \in (0,1),$ of problem (\ref{p}).

The proof of Theorem \ref{T2} is done in Section \ref{S3}. It is based on
topological degree theory with suitable truncations. Here, it suffices to
show the existence of a second (positive) solution for problem (\ref{p}).
The first one is given by Theorem \ref{T1} which is located in a rectangle
formed by the sub-supersolutions. However, due to the singular terms in
system (\ref{p}), the degree theory cannot be directly implemented. To
handle this difficulty, the degree calculation is applied for the
regularized problem (\ref{pr}) for $\varepsilon >0$. Under assumption (\ref%
{h1}), Theorem \ref{T1} ensures the existence of a smooth solution for (\ref%
{p}). This gives rise to the possible existence a constant $R>0$ such that
all solutions $(u,v)$ with $C^{1,\gamma }$-regularity satisfy $\left\Vert
u\right\Vert _{C^{1,\gamma }},\left\Vert v\right\Vert _{C^{1,\gamma }}<R.$
On the basis of this, we show that the degree of an operator corresponding
to system (\ref{pr}) on a larger set is $0$. Another hand, we show that the
degree of an operator corresponding to the system (\ref{pr}) is $1$ on an
appropriate set. This leads to the existence of a second solution for (\ref%
{pr}) by using the excision property of Leray-Schauder degree. Then the
existence of a second solution for (\ref{p}) is derived by passing to the
limit as $\varepsilon \rightarrow 0$.

In what follows, we denote by $\phi _{1,p}$ and $\phi _{1,q}$ the normalized
positive eigenfunctions associated with the principal eigenvalues $\lambda
_{1,p}$ and $\lambda _{1,q}$ of $-\Delta _{p}$ and $-\Delta _{q}$,
respectively: 
\begin{equation}
\begin{array}{c}
-\Delta _{p}\phi _{1,p}=\lambda _{1,p}\left\vert \phi _{1,p}\right\vert
^{p-2}\phi _{1,p}\text{ \ in }\Omega ,\text{ \ }\phi _{1,p}=0\text{ \ on }%
\partial \Omega ,\text{ \ }\int_{\Omega }\phi _{1,p}^{p}=1%
\end{array}
\label{6}
\end{equation}%
and%
\begin{equation}
\begin{array}{c}
-\Delta _{q}\phi _{1,q}=\lambda _{1,q}\left\vert \phi _{1,q}\right\vert
^{q-2}\phi _{1,q}\text{ \ in }\Omega ,\text{ \ }\phi _{1,q}=0\text{ \ on }%
\partial \Omega ,\text{ \ }\int_{\Omega }\phi _{1,q}^{q}=1.%
\end{array}
\label{8}
\end{equation}

The strong maximum principle ensures the existence of positive constants $%
l_{1}$ and $l_{2}$ such that%
\begin{equation}
l_{1}\phi _{1,p}(x)\leq \phi _{1,q}(x)\leq l_{2}\phi _{1,p}(x)\text{ for all 
}x\in \Omega .  \label{5}
\end{equation}%
For a later use we recall that there exists a constant $l>0$ such that%
\begin{equation}
\phi _{1,p}(x),\phi _{1,q}(x)\geq ld(x)\text{ for all }x\in \Omega ,
\label{67}
\end{equation}%
where $d(x):=dist(x,\partial \Omega )$ (see, e.g., \cite{GST}). Moreover,
since $\phi_{1,p}$ and $\phi_{1,q}$ belongs to $C^{1}(\overline{\Omega})$,
there is $M>0$ such that 
\begin{equation}
\begin{array}{c}
M=\underset{x \in \overline{\Omega} }{\max }\{|\phi _{1,p}(x)|+|\phi
_{1,q}(x)|\}.%
\end{array}
\label{23}
\end{equation}

\section{Proof of Theorem \protect\ref{T1}: Existence of the first solution}

\label{S2}

Let us define $w_{1}$ and $w_{2}$ as the unique weak solutions of the
problems 
\begin{equation}
\left\{ 
\begin{array}{ll}
-\Delta _{p}w_{1}=w_{1}^{\alpha _{1}} & \text{in }\Omega , \\ 
w_{1}>0 & \text{in }\Omega , \\ 
w_{1}=0 & \text{on }\partial \Omega%
\end{array}%
\right. \text{ \ and \ }\left\{ 
\begin{array}{ll}
-\Delta _{q}w_{2}=w_{2}^{\beta _{2}} & \text{in }\Omega , \\ 
w_{2}>0 & \text{in }\Omega , \\ 
w_{2}=0 & \text{on }\partial \Omega ,%
\end{array}%
\right.  \label{20}
\end{equation}%
respectively, which are known to satisfy 
\begin{equation}
c_{2}\phi _{1,p}(x)\leq w_{1}(x)\leq c_{3}\phi _{1,p}(x)\text{ \ and \ }%
c_{2}^{\prime }\phi _{1,q}(x)\leq w_{2}(x)\leq c_{3}^{\prime }\phi _{1,q}(x),
\label{21}
\end{equation}%
with positive constants $c_{2},c_{3},c_{2}^{\prime },c_{3}^{\prime }$ (see 
\cite{GST}). Consider $\xi _{1},\xi _{2}\in C^{1}\left( \overline{\Omega }%
\right) $ the solutions of the homogeneous Dirichlet problems:%
\begin{equation}
\left\{ 
\begin{array}{ll}
-\Delta _{p}\xi _{1}(x)=\phi _{1,p}^{\alpha _{1}}(x) & \text{ in }\Omega ,
\\ 
\xi _{1}=0 & \text{ on }\partial \Omega%
\end{array}%
\right. ,\text{ }\left\{ 
\begin{array}{ll}
-\Delta _{q}\xi _{2}(x)=\phi _{1,q}^{\beta _{2}}(x) & \text{ in }\Omega , \\ 
\xi _{2}=0 & \text{\ on }\partial \Omega .%
\end{array}%
\right.  \label{12}
\end{equation}%
The Hardy--Sobolev inequality (see, e.g., \cite[Lemma 2.3]{AC}) guarantees
that the right-hand side of (\ref{12}) belongs to $W^{-1,p^{\prime }}(\Omega
)$ and $W^{-1,q^{\prime }}(\Omega )$, respectively. Consequently, the
Minty--Browder theorem (see \cite[Theorem V.15]{B}) implies the existence of
unique $\xi _{1}$ and $\xi _{2}$ in (\ref{12}). Moreover, (\ref{20}), (\ref%
{21}), the monotonicity of the operators $-\Delta _{p}$ and $-\Delta _{q}$
yield 
\begin{equation}
c_{0}\phi _{1,p}(x)\leq \xi _{1}(x)\leq c_{1}\phi _{1,p}(x)\text{ and }%
c_{0}^{\prime }\phi _{1,q}(x)\leq \xi _{2}(x)\leq c_{1}^{\prime }\phi
_{1,q}(x)\text{ in }\Omega ,  \label{36}
\end{equation}%
for some positive constants $c_{0},c_{1},c_{0}^{\prime },c_{1}^{\prime }$.
Let $z_{1}$ and $z_{2}$ satisfy%
\begin{equation}
-\Delta _{p}z_{1}(x)=h_{1}(x),\text{ }z_{1}=0\text{ \ on }\partial \Omega ,
\label{1}
\end{equation}%
and%
\begin{equation}
-\Delta _{q}z_{2}(x)=h_{2}(x),\text{ }z_{2}=0\text{ \ on }\partial \Omega 
\text{,}  \label{2}
\end{equation}%
where 
\begin{equation}
h_{1}(x)=\left\{ 
\begin{array}{ll}
\phi _{1,p}^{\alpha _{1}}(x) & \text{in \ }\Omega \backslash \overline{%
\Omega }_{\delta }, \\ 
-\phi _{1,p}^{\alpha _{1}}(x) & \text{in \ }\Omega _{\delta },%
\end{array}%
\right.  \label{h1*}
\end{equation}

\begin{equation}
h_{2}(x)=\left\{ 
\begin{array}{ll}
\phi _{1,q}^{\beta _{2}}(x) & \text{in \ }\Omega \backslash \overline{\Omega 
}_{\delta }, \\ 
-\phi _{1,q}^{\beta _{2}}(x) & \text{in \ }\Omega _{\delta }%
\end{array}%
\right.  \label{h2*}
\end{equation}%
and 
\begin{equation*}
\Omega _{\delta }=\left\{ x\in \Omega :d(x)<\delta \right\} ,
\end{equation*}%
with a fixed $\delta >0$ sufficiently small and $d(x)=d\left( x,\partial
\Omega \right) $.

The Hardy-Sobolev inequality together with the Minty-Browder theorem imply
the existence and uniqueness of $z_{1}$ and $z_{2}$ in (\ref{1}) and (\ref{2}%
). Moreover, (\ref{1}) and (\ref{2}), the monotonicity of the operators $%
-\Delta _{p}$ and $-\Delta _{q} $ and \cite[Corollary 3.1]{Hai} imply that 
\begin{equation}
\begin{array}{l}
\frac{c_{0}}{2}\phi _{1,p}(x)\leq z_{1}(x)\leq c_{1}\phi _{1,p}(x)\text{ and 
}\frac{c_{0}^{\prime }}{2}\phi _{1,q}(x)\leq z_{2}(x)\leq c_{1}^{\prime
}\phi _{1,q}(x)\text{ in }\Omega .%
\end{array}
\label{c2}
\end{equation}

\vspace{0.2cm} Next, our goal is to show the existence of sub and
supersolution for $(P)$. \newline

\noindent \textbf{Existence of subsolution:} \newline

For a constant $C>0$, we have 
\begin{equation}
\begin{array}{l}
-C^{-(p-1)}\phi _{1,p}^{\alpha _{1}}(x)<0\leq (C^{-1}z_{1}(x))^{\alpha
_{1}}(C^{-1}z_{2}(x))^{\beta _{1}},\text{ }x\in \Omega _{\delta }%
\end{array}
\label{18}
\end{equation}%
and%
\begin{equation}
\begin{array}{c}
-C^{-(q-1)}\phi _{1,q}^{\beta _{2}}(x)<0\leq (C^{-1}z_{1}(x))^{\alpha
_{2}}(C^{-1}z_{2}(x))^{\beta _{2}},\text{ }x\in \Omega _{\delta }.%
\end{array}
\label{19}
\end{equation}%
Let $\mu >0$ be a constant such that%
\begin{equation}
\begin{array}{c}
\phi _{1}\left( x\right) ,\phi _{2}\left( x\right) \geq \mu \text{ in }%
\Omega \backslash \overline{\Omega }_{\delta }.%
\end{array}
\label{256}
\end{equation}%
Then, since $\alpha _{1}<0<\beta _{1}$, (\ref{c2}) and (\ref{256}) lead to 
\begin{equation}
\begin{array}{l}
C^{\alpha _{1}+\beta _{1}-(p-1)}\phi _{1,p}^{\alpha
_{1}}(x)(z_{1}(x))^{-\alpha _{1}}\leq C^{\alpha _{1}+\beta _{1}-(p-1)}\phi
_{1,p}^{\alpha _{1}}(x)(c_{1}\phi _{1,p}(x))^{-\alpha _{1}} \\ 
=C^{\alpha _{1}+\beta _{1}-(p-1)}(Mc_{1})^{-\alpha _{1}}<(c_{0}^{\prime }\mu
)^{\beta _{1}}\leq (c_{0}^{\prime }\phi _{1,q}(x))^{\beta _{1}} \\ 
\leq (z_{2}\left( x\right) )^{\beta _{1}},\ \ \text{for all }x\in \Omega
\backslash \overline{\Omega }_{\delta },%
\end{array}
\label{28}
\end{equation}%
provided $C>0$ large enough. This is equivalent to%
\begin{equation}
\begin{array}{l}
C^{-(p-1)}\phi _{1,p}^{\alpha _{1}}(x)<(C^{-1}z_{1}(x))^{\alpha
_{1}}(C^{-1}z_{2}\left( x\right) )^{\beta _{1}},\ \ \text{for all }x\in
\Omega \backslash \overline{\Omega }_{\delta }.%
\end{array}
\label{29*}
\end{equation}%
Similarly, 
\begin{equation}
\begin{array}{c}
C^{-(q-1)}\phi _{1,q}^{\beta _{2}}(x)<\left( C^{-1}z_{1}\left( x\right)
\right) ^{\alpha _{2}}(C^{-1}z_{2}(x))^{\beta _{2}}\ \ \text{for all }x\in
\Omega \backslash \overline{\Omega }_{\delta },%
\end{array}
\label{29}
\end{equation}%
for $C>0$ large enough. The pair 
\begin{equation}
\begin{array}{c}
\left( \underline{u},\underline{v}\right) =C^{-1}\left( z_{1},z_{2}\right) .%
\end{array}
\label{30}
\end{equation}%
is a subsolution for $(P)$, Indeed, a direct computation shows that 
\begin{equation}
\begin{array}{c}
\int_{\Omega }\left\vert \nabla \underline{u}\right\vert ^{p-2}\nabla 
\underline{u}\nabla \varphi \text{ }dx=C^{-(p-1)}\int_{\Omega \backslash
\Omega _{\delta }}\phi _{1,p}^{\alpha _{1}}\varphi \text{ }%
dx-C^{-(p-1)}\int_{\Omega _{\delta }}\phi _{1,p}^{\alpha _{1}}\varphi \text{ 
}dx%
\end{array}
\label{*}
\end{equation}%
and%
\begin{equation}
\begin{array}{c}
\int_{\Omega }\left\vert \nabla \underline{v}\right\vert ^{q-2}\nabla 
\underline{v}\nabla \psi =C^{-(q-1)}\int_{\Omega \backslash \Omega _{\delta
}}\phi _{1,q}^{\beta _{2}}\psi \text{ }dx-C^{-(q-1)}\int_{\Omega _{\delta
}}\phi _{1,q}^{\beta _{2}}\psi \text{ }dx,%
\end{array}
\label{**}
\end{equation}%
where $\left( \varphi ,\psi \right) \in W_{0}^{1,p}\left( \Omega \right)
\times W_{0}^{1,q}\left( \Omega \right) $ with $\varphi ,\psi \geq 0$.
Combining (\ref{*}), (\ref{**}), (\ref{18}), (\ref{19}), (\ref{28}) and (\ref%
{29}), it is readily seen that%
\begin{equation*}
\begin{array}{c}
\int_{\Omega }\left\vert \nabla \underline{u}\right\vert ^{p-2}\nabla 
\underline{u}\nabla \varphi \leq \int_{\Omega }\underline{u}^{\alpha _{1}}%
\underline{v}^{\beta _{1}}\varphi%
\end{array}%
\end{equation*}%
and%
\begin{equation*}
\begin{array}{c}
\int_{\Omega }\left\vert \nabla \underline{v}\right\vert ^{q-2}\nabla 
\underline{v}\nabla \psi \leq \int_{\Omega }\underline{u}^{\alpha _{2}}%
\underline{v}^{\beta _{2}}\psi ,%
\end{array}%
\end{equation*}%
for all $\left( \varphi ,\psi \right) \in W_{0}^{1,p}\left( \Omega \right)
\times W_{0}^{1,q}\left( \Omega \right) $ with $\varphi ,\psi \geq 0$. This
proves that $\left( \underline{u},\underline{v}\right) $ is a subsolution
for $(P)$.\newline

\noindent \textbf{Existence of supersolution:} \newline

Next, we prove that 
\begin{equation}
(\overline{u},\overline{v})=C(\xi _{1},\xi _{2})  \label{32}
\end{equation}%
is a supersolution for problem (\ref{p}) for $C>0$ large enough. Obviously,
we have $\left( \overline{u},\overline{v}\right) \geq \left( \underline{u},%
\underline{v}\right) $ in $\overline{\Omega }$ for $C$ large enough. Taking
into account (\ref{12}), (\ref{36}), (\ref{23}) and (\ref{h1}) we derive
that in $\overline{\Omega }$ one has the estimates 
\begin{equation*}
\begin{array}{l}
\overline{u}^{-\alpha _{1}}\overline{v}^{-\beta _{1}}(-\Delta _{p}\overline{u%
})=C^{p-1-\alpha _{1}-\beta _{1}}\xi _{2}^{-\beta _{1}}{\geq }C^{p-1-\alpha
_{1}-\beta _{1}}(c_{1}^{\prime }\phi _{1,q}(x))^{-\beta _{1}} \\ 
\geq C^{p-1-\alpha _{1}-\beta _{1}}(c_{1}^{\prime }M)^{-\beta _{1}}\geq 1%
\text{ in }\overline{\Omega }%
\end{array}%
\end{equation*}%
and%
\begin{equation*}
\overline{u}^{-\alpha _{2}}\overline{v}^{-\beta _{2}}(-\Delta _{q}\overline{v%
}){\geq }C^{q-1-\alpha _{2}-\beta _{2}}(c_{1}M)^{-\alpha _{2}}\geq 1\text{
in }\overline{\Omega },
\end{equation*}%
provided that $C>0$ is sufficiently large. Consequently, it turns out that%
\begin{equation}
\int_{\Omega }\left\vert \nabla \overline{u}\right\vert ^{p-2}\nabla 
\overline{u}\nabla \varphi \text{ }dx\geq \int_{\Omega }\overline{u}^{\alpha
_{1}}\overline{v}^{\beta _{1}}\varphi \text{ }dx  \label{33}
\end{equation}%
and 
\begin{equation}
\int_{\Omega }\left\vert \nabla \overline{v}\right\vert ^{q-2}\nabla 
\overline{v}\nabla \psi \text{ }dx\geq \int_{\Omega }\overline{u}^{\alpha
_{2}}\overline{v}^{\beta _{2}}\psi \text{ }dx,  \label{34}
\end{equation}%
for all $\left( \varphi ,\psi \right) \in W_{0}^{1,p}\left( \Omega \right)
\times W_{0}^{1,q}\left( \Omega \right) .$ \newline

\noindent \textbf{Proof of Theorem \ref{T1} (conclusion): } \newline

Using (\ref{h1}), (\ref{67}), (\ref{c}), (\ref{30}), (\ref{32}), (\ref{c2})
and (\ref{36}), we get%
\begin{equation*}
u^{\alpha _{1}}v^{\beta _{1}}\leq \underline{u}^{\alpha _{1}}\overline{v}%
^{\beta _{1}}\leq C_{1}d(x)^{\alpha _{1}}\text{ for all }x\in \Omega
\end{equation*}%
and%
\begin{equation*}
u^{\alpha _{2}}v^{\beta _{2}}\leq \overline{u}^{\alpha _{2}}\underline{v}%
^{\beta _{2}}\leq C_{2}d(x)^{\beta _{2}}\text{ for all }x\in \Omega ,
\end{equation*}%
where $C_{1}$ and $C_{2}$ are positive constants. Then, owing to \cite[%
Theorem 2]{KM} we deduce that there exists a solution $(u,v)\in C^{1,\gamma
}(\overline{\Omega })\times C^{1,\gamma }(\overline{\Omega }),$ for some $%
\gamma \in (0,1),$ of problem (\ref{p}) within $\left[ \underline{u},%
\overline{u}\right] \times \left[ \underline{v},\overline{v}\right] $. This
complete the proof.

\section{Proof of Theorem \protect\ref{T2}}

\label{S3}

According to Theorem \ref{T1} we know that problem (\ref{p}) possesses a
(positive) solution $(u,v)$ in $C^{1,\gamma }(\overline{\Omega })\times
C^{1,\gamma }(\overline{\Omega }),$ located in the rectangle $[\underline{u},%
\overline{u}]\times \lbrack \underline{v},\overline{v}]$ for certain $\gamma
\in (0,1)$. Thus, to prove Theorem \ref{T2} it suffices to show the
existence of a second solution for problem (\ref{p}).

Before starting the proof of Theorem \ref{T2}, we would like point out that
by Theorem \ref{T1} the set of solutions $(u,v)$ in $C^{1,\gamma }(\overline{%
\Omega })\times C^{1,\gamma }(\overline{\Omega }),$ $\gamma \in (0,1),$ for
problem (\ref{p}) is not empty. Then, without any loss of generality, we may
assume that there is a constant $R>0$ such that all solutions $(u,v)$ with $%
C^{1,\gamma }$-regularity satisfy 
\begin{equation}
\left\Vert u\right\Vert _{C^{1,\gamma }(\overline{\Omega })},\left\Vert
v\right\Vert _{C^{1,\gamma }\overline{\Omega }}<R.  \label{15}
\end{equation}
Otherwise, there are infinity solutions with $C^{1,\gamma }$-regularity and
the proof of Theorem \ref{T2} is completed.

Hereafter, we denote 
\begin{equation*}
B_{R}(0)=\left\{ (u,v)\in C^{1}(\overline{\Omega })\times C^{1}(\overline{%
\Omega })\,:\,\Vert u\Vert _{C^{1}}+\Vert v\Vert _{C^{1}}<R\right\} ,
\end{equation*}%
\begin{equation*}
\mathcal{O}_{R}=\left\{ (u,v)\in B_{R}(0)\,:\,\underline{u}\ll u\ll R\,\,\,%
\mbox{and}\,\,\,\underline{v}\ll v\ll R\right\}
\end{equation*}%
and%
\begin{equation*}
\mathcal{\hat{O}}=\left\{ (u,v)\in B_{R}(0)\,:\,\underline{u}\ll u\ll \hat{u}%
\,\,\,\mbox{and}\,\,\,\underline{v}\ll v\ll \hat{v}\right\} ,
\end{equation*}%
where 
\begin{equation}
(\hat{u},\hat{v})=\Lambda (w_{1},w_{2})  \label{3}
\end{equation}%
with $w_{1},w_{2}$ fixed in (\ref{20}) and $\Lambda >0$ is a constant which
will be chosen later on. A simple computation gives that $\mathcal{O}_{R}$
and $\mathcal{\hat{O}}$ are open sets in $C^{1}(\overline{\Omega })\times
C^{1}(\overline{\Omega })$.

In what follows, we will assume without loss of generality that 
\begin{equation*}
R>\max \{\Vert \underline{u}\Vert _{\infty },\Vert \overline{u}\Vert
_{\infty },\Vert \underline{v}\Vert _{\infty },\Vert \overline{v}\Vert
_{\infty },\Vert \hat{u}\Vert _{\infty },\Vert \hat{v}\Vert _{\infty }\}.
\end{equation*}%
In the sequel, we use the notation $u_{1}\ll u_{2}$ when $u_{1},u_{2}\in
C^{1}(\overline{\Omega })$ satisfy: 
\begin{equation*}
\begin{array}{c}
u_{1}(x)<u_{2}(x)\,\,\,\forall x\in \Omega \,\,\,\mbox{and}\,\,\,\frac{%
\partial u_{2}}{\partial \nu }<\frac{\partial u_{1}}{\partial \nu }\,\,\,%
\mbox{on}\,\,\,\partial \Omega ,%
\end{array}%
\end{equation*}%
where $\nu $ is the outward normal to $\partial \Omega $.

The next proposition is useful for proving our second main result.

\begin{proposition}
\label{P1}Assume (\ref{h1}) holds. Then all solutions $(u,v)$ of (\ref{p})
within $[\underline{u},\overline{u}]\times \lbrack \underline{v},\overline{v}%
]$ verifies 
\begin{equation}
\begin{array}{c}
u(x)\ll \hat{u}(x)\text{ \ and \ }v(x)\ll \hat{v}(x)\text{ \ in }\Omega .%
\end{array}
\label{c3}
\end{equation}
\end{proposition}

\begin{proof}
From (\ref{32}), (\ref{c3}), (\ref{h1}), (\ref{36}), (\ref{67}), (\ref{20})
and (\ref{21}), it follows that%
\begin{equation}
\begin{array}{l}
-\Delta _{p}u=u^{\alpha _{1}}v^{\beta _{1}}\leq \underline{u}^{\alpha _{1}}%
\overline{v}^{\beta _{1}}\leq (C^{-1}\frac{c_{0}}{2}\phi _{1,p})^{\alpha
_{1}}(Cc_{1}^{\prime }\phi _{1,q})^{\beta _{1}} \\ 
\leq C^{-\alpha _{1}+\beta _{1}}(\frac{c_{0}}{2})^{\alpha
_{1}}(c_{1}^{\prime }M)^{\beta _{1}}\phi _{1,p}^{\alpha _{1}}\leq C^{-\alpha
_{1}+\beta _{1}}(\frac{c_{0}}{2})^{\alpha _{1}}(c_{1}^{\prime }M)^{\beta
_{1}}(c_{3}w_{1})^{\alpha _{1}} \\ 
< \Lambda ^{p-1}w_{1}^{\alpha _{1}}=-\Delta _{p}(\Lambda w_{1})=-\Delta _{p}%
\hat{u}\text{ in }\Omega ,%
\end{array}
\label{69}
\end{equation}%
provided that $\Lambda $ is large enough. Proceeding in the same way with
the second equation in (\ref{p}) results in%
\begin{equation}
\begin{array}{c}
-\Delta _{q}v < -\Delta _{q}(\Lambda \xi _{2})=-\Delta _{q}\hat{v}\text{\ \
in }\Omega ,%
\end{array}%
\end{equation}%
for $\Lambda $ large enough. Consequently, the strong comparison principle
found in \cite[Proposition 2.6]{AR} leads to the conclusion. This ends the
proof.
\end{proof}

\subsection{An auxiliary problem}

In this subsection, we will use the Topological degree to get the second
solution. However, the singular terms in system (\ref{p}) prevents the
degree calculation to be well defined. To overcome this difficulty, we
disturb system (\ref{p}) by introducing a parameter $\varepsilon \in (0,1)$.
This gives rise to a regularized system for (\ref{p}) defined for $%
\varepsilon >0$ as follows:%
\begin{equation}
\left\{ 
\begin{array}{ll}
-\Delta _{p}u=\left( u+\varepsilon \right) ^{\alpha _{1}}v^{\beta _{1}} & 
\text{ in }\Omega , \\ 
-\Delta _{q}v=u^{\alpha _{2}}\left( v+\varepsilon \right) ^{\beta _{2}} & 
\text{ in }\Omega , \\ 
u(x),v(x)>0 & \text{ in } \Omega, \\ 
u,v=0 & \text{ on }\partial \Omega .%
\end{array}%
\right.  \tag{$P_{r}$}  \label{pr}
\end{equation}%
We apply the degree theory for the regularized problem (\ref{pr}). This
leads to find a positive solution for (\ref{pr}) lying outside of the set $%
\mathcal{\hat{O}}$. Then the existence of a second solution of (\ref{p}) is
obtain by passing to the limit in (\ref{pr}) as $\varepsilon \rightarrow 0$.
The proof comprises four steps.

\begin{remark}
\label{R3}It is very important to observe that the same reasoning exploited
in the proof of Theorem \ref{T1} and Proposition \ref{P1} furnishes that
problem (\ref{pr}) has a (positive) solution $(u_{\varepsilon
},v_{\varepsilon })\in C^{1,\gamma }(\overline{\Omega })\times C^{1,\gamma }(%
\overline{\Omega }),$ $\gamma \in (0,1)$, within $\left[ \underline{u},%
\overline{u}\right] \times \left[ \underline{v},\overline{v}\right] ,$ where
functions $(\underline{u},\underline{v})$ and $(\overline{u},\overline{v})$
are sub-supersolutions of (\ref{pr}) and $(u_{\varepsilon },v_{\varepsilon
}) $ verifies%
\begin{equation*}
\begin{array}{c}
u_{\varepsilon }(x)\ll \hat{u}(x)\text{ \ and \ }v_{\varepsilon }(x)\ll \hat{%
v}(x)\text{ \ in }\Omega \text{,}%
\end{array}%
\end{equation*}%
for all $\varepsilon \in (0,1)$.
\end{remark}

\vspace{0.5 cm}

\noindent \textbf{Topological degree: The first estimate.}

\bigskip

We transform the problem (\ref{pr}) to one with helpful monotonicity
properties. To this end, let us introduce the functions 
\begin{equation}
\widetilde{\phi }=\left\{ 
\begin{array}{l}
R\text{ \ if }\phi \geq R \\ 
\phi \text{ \ if }\underline{u}\leq \phi \leq R \\ 
\underline{u}\text{ \ if }\phi \leq \underline{u}%
\end{array}%
\right. ,\text{ \ }\widetilde{\varphi }=\left\{ 
\begin{array}{l}
R\text{ \ if }\varphi \geq R \\ 
\phi \text{ \ if }\underline{v}\leq \varphi \leq R \\ 
\underline{v}\text{ \ if }\varphi \leq \underline{v},%
\end{array}%
\right.  \label{5**}
\end{equation}%
where $(\underline{u},\underline{v})$ and $R$ are given by (\ref{30}) and (%
\ref{15}), respectively. Define the operators%
\begin{equation*}
\begin{array}{c}
T_{p,\varepsilon }(u)=-\Delta _{p}u+\rho \max \{(\underline{u}+\varepsilon
)^{\alpha _{1}-1}R^{\beta _{1}},u^{p-1}\}, \\ 
T_{q,\varepsilon }(v)=-\Delta _{q}v+\rho \max \{R^{\alpha _{2}}(\underline{v}%
+\varepsilon )^{\beta _{2}-1},v^{q-1}\},%
\end{array}%
\end{equation*}%
for $t\in \lbrack 0,1]$, $\varepsilon \in (0,1)$ and a constant $\rho >0$.
We shall study the homotopy class of problem%
\begin{equation}
\left\{ 
\begin{array}{ll}
T_{p,\varepsilon }(u)={f_{1,\varepsilon ,t}(x,}\widetilde{{u}}{,}\widetilde{{%
v}}{)} & \text{ in }\Omega , \\ 
T_{q,\varepsilon }(v)={f_{2,\varepsilon ,t}(x,\widetilde{{u}}{,}\widetilde{{v%
}})} & \text{ in }\Omega , \\ 
u,v>0\text{ \ in }\Omega , &  \\ 
u,v=0\text{ \ on }\partial \Omega , & 
\end{array}%
\right.  \tag{$P_{f}$}  \label{50}
\end{equation}%
where functions ${f_{1,\varepsilon ,t}}$ and ${f_{2,\varepsilon ,t}}$ are
defined as follows:%
\begin{equation}
\begin{array}{c}
{f_{1,\varepsilon ,t}(x,\widetilde{{u}}{,}\widetilde{{v}})}=t(\widetilde{u}%
+\varepsilon )^{\alpha _{1}}\widetilde{v}^{\beta _{1}}+m(1-t)\widetilde{u}%
^{p-1} \\ 
+\rho \max \{(\underline{u}+\varepsilon )^{\alpha _{1}-1}R^{\beta _{1}},%
\widetilde{u}^{p-1}\},%
\end{array}
\label{15**}
\end{equation}%
\begin{equation}
\begin{array}{c}
{f_{2,\varepsilon ,t}(x,\widetilde{{u}}{,}\widetilde{{v}})}=t\widetilde{u}%
^{\alpha _{2}}(\widetilde{v}+\varepsilon )^{\beta _{2}}+m(1-t)\widetilde{v}%
^{q-1} \\ 
+\rho \max \{R^{\alpha _{2}}(\underline{v}+\varepsilon )^{\beta _{2}-1},%
\widetilde{v}^{q-1}\},%
\end{array}
\label{15*}
\end{equation}%
with a constant $m>\max \{\lambda _{1,p},\lambda _{1,q}\}$. In the sequel,
we fix the constant $\rho >0$ in (\ref{50}) sufficiently large so that the
following inequalities are satisfied: 
\begin{equation*}
\begin{array}{c}
t\alpha _{1}(s_{1}+\varepsilon )^{\alpha _{1}-1}s_{2}^{\beta _{1}}+\rho \max
\{(\underline{u}+\varepsilon )^{\alpha _{1}-1}R^{\beta
_{1}},(p-1)s_{1}^{p-2}\}\geq 0%
\end{array}%
\end{equation*}%
and 
\begin{equation*}
\begin{array}{c}
t\beta _{2}(s_{2}+\varepsilon )^{\beta _{2}-1}s_{1}^{\alpha _{2}}+\rho \max
\{R^{\alpha _{2}}(\underline{v}+\varepsilon )^{\beta
_{2}-1},(q-1)s_{2}^{q-2}\}\geq 0,%
\end{array}%
\end{equation*}%
uniformly in $x\in \Omega ,$ for $(s_{1},s_{2})\in \lbrack \underline{u}%
,R]\times \lbrack \underline{v},R],$ $\varepsilon \in (0,1)$. By the above
choice of $\rho $, the term in the right-hand side of first (resp. second)
equation in (\ref{50}) increases as $u$ (resp. $v$) increases, for all $%
\varepsilon >0$ small.

\vspace{0.5 cm}

The next result is crucial in our approach, because it establishes an
important prior estimate for system (\ref{50}). Moreover, it is also shown
that the solutions of problem (\ref{50}) cannot occur outside the rectangle
formed by the subsolution $(\underline{u},\underline{v})$ and the constant $%
R $.

\begin{proposition}
\label{P2}Assume (\ref{h1}) holds. If $(u,v)$ is a solution of (\ref{50}),
then $(u,v)$ belongs to $C^{1,\gamma }(\overline{\Omega })\times C^{1,\gamma
}(\overline{\Omega })$ for some $\gamma \in (0,1)$ and satisfies%
\begin{equation}
\left\Vert u\right\Vert _{C^{1,\gamma }(\overline{\Omega })},\left\Vert
v\right\Vert _{C^{1,\gamma }(\overline{\Omega })}<R.  \label{31}
\end{equation}%
Moreover, it holds%
\begin{equation}
\begin{array}{c}
\underline{u}(x)\ll u(x)\text{ \ and \ }\underline{v}(x)\ll v(x)\text{ \ in }%
\Omega ,\text{ \ }\forall t\in \lbrack 0,1].%
\end{array}
\label{16}
\end{equation}
\end{proposition}

\begin{proof}
First, by Moser's iterations technique, we prove the boundedness for
solutions of (\ref{50}) in $L^{\infty }(\Omega )\times L^{\infty }(\Omega )$%
. Assuming (\ref{16}) holds, it follows that%
\begin{equation}
\begin{array}{l}
\max \{(\underline{u}+\varepsilon )^{\alpha _{1}}R^{\beta
_{1}},u^{p-1}\}-\max \{(\underline{u}+\varepsilon )^{\alpha _{1}}R^{\beta
_{1}},\widetilde{u}^{p-1}\}\geq 0\text{ in }\Omega%
\end{array}
\label{68}
\end{equation}%
and%
\begin{equation}
\begin{array}{l}
\max \{R^{\alpha _{2}}(\underline{v}+\varepsilon )^{\beta
_{2}},v^{q-1}\}-\max \{R^{\alpha _{2}}(\underline{v}+\varepsilon )^{\beta
_{2}},\widetilde{v}^{q-1}\}\geq 0\text{ in }\Omega .%
\end{array}
\label{68*}
\end{equation}%
Then, 
\begin{equation}
\left\{ 
\begin{array}{l}
-\Delta _{p}u\leq \widetilde{u}^{\alpha _{1}}\widetilde{v}^{\beta _{1}}+m%
\widetilde{u}^{p-1}\text{ in }\Omega , \\ 
-\Delta _{q}v\leq \widetilde{u}^{\alpha _{2}}\widetilde{v}^{\beta _{2}}+m%
\widetilde{v}^{q-1}\text{ in }\Omega , \\ 
u,v>0\text{ in }\Omega , \\ 
u,v=0\text{ on }\partial \Omega .%
\end{array}%
\right.  \label{500}
\end{equation}%
Given a constant $A\in (0,R]$, define on $\Omega $ the functions%
\begin{equation*}
\begin{array}{l}
u_{A}=\min \{u(x),A\}\text{ \ and \ }v_{A}=\min \{v(x),A\}.%
\end{array}%
\end{equation*}%
Acting on (\ref{50}) with 
\begin{equation*}
\begin{array}{c}
\left( \varphi ,\psi \right) =\left( u_{A}^{k_{1}p+1},v_{A}^{\overline{k}%
_{1}q+1}\right) ,%
\end{array}%
\end{equation*}%
where 
\begin{equation}
\begin{array}{c}
\left( k_{1}+1\right) p=p^{\ast }\text{ and }\left( \overline{k}%
_{1}+1\right) q=q^{\ast },%
\end{array}%
\end{equation}%
and integrating over $\Omega $ we get 
\begin{equation}
\begin{array}{l}
\left( k_{1}p+1\right) \int_{\Omega }\left\vert \nabla u_{A}\right\vert
^{p}u_{A}^{k_{1}p}\text{ }dx\leq \int_{\Omega }(\widetilde{u}^{\alpha _{1}}%
\widetilde{v}^{\beta _{1}}+m\widetilde{u}^{p-1})u_{A}^{k_{1}p}\text{ }dx%
\end{array}
\label{36*}
\end{equation}%
and%
\begin{equation}
\begin{array}{c}
\left( \overline{k}_{1}q+1\right) \int_{\Omega }\left\vert \nabla
v_{A}\right\vert ^{q}v_{A}^{\overline{k}_{1}q}\text{ }dx\leq \int_{\Omega }(%
\widetilde{u}^{\alpha _{2}}\widetilde{v}^{\beta _{2}}+m\widetilde{v}%
^{q-1})v_{A}^{\overline{k}_{1}q+1}\text{ }dx.%
\end{array}
\label{37}
\end{equation}%
By the Sobolev embedding theorem, the left-hand sides of (\ref{36*}) and (%
\ref{37}) are estimated from below as follows 
\begin{equation}
\begin{array}{l}
(k_{1}p+1)\int_{\Omega }|\nabla u_{A}|^{p}u_{A}^{k_{1}p}=\frac{k_{1}p+1}{%
\left( k_{1}+1\right) ^{p}}\int_{\Omega }|\nabla u_{A}^{k_{1}+1}|^{p}\geq
C_{1}\frac{\left( k_{1}p+1\right) }{\left( k_{1}+1\right) ^{p}}\left\Vert
u_{A}\right\Vert _{\left( k_{1}+1\right) p^{\ast }}^{p^{\ast }}%
\end{array}
\label{l6}
\end{equation}%
and 
\begin{equation}
\begin{array}{l}
\left( \overline{k}_{1}q+1\right) \int_{\Omega }\left\vert \nabla
v_{A}\right\vert ^{q}v_{A}^{\overline{k}_{1}q}=\frac{\left( \overline{k}%
_{1}q+1\right) }{\left( \overline{k}_{1}+1\right) ^{q}}\int_{\Omega }|\nabla
v_{A}^{\overline{k}_{1}+1}|^{q}\geq C_{1}^{\prime }\frac{\left( \overline{k}%
_{1}q+1\right) }{\left( \overline{k}_{1}+1\right) ^{q}}\left\Vert
v_{A}\right\Vert _{\left( \overline{k}_{1}+1\right) q^{\ast }}^{q^{\ast }},%
\end{array}
\label{l7}
\end{equation}%
where $C_{1}$ and $C_{1}^{\prime }$ are some positive constants. By noticing
that $k_{1}p+1+\alpha _{1}>0$ and $\overline{k}_{1}q+1+\beta _{2}>0$ it
turns out that%
\begin{equation}
\begin{array}{l}
\int_{\Omega }(\widetilde{u}^{\alpha _{1}}\widetilde{v}^{\beta _{1}}+m%
\widetilde{u}^{p-1})u_{A}^{k_{1}p+1}\text{ }dx\leq \int_{\Omega
}u_{A}^{\alpha _{1}+k_{1}p+1}v^{\beta _{1}}\text{ }dx+m\int_{\Omega
}u^{(k_{1}+1)p}\text{ }dx \\ 
\leq \int_{\Omega }u^{\alpha _{1}+k_{1}p+1}v^{\beta _{1}}\text{ }%
dx+m\int_{\Omega }u^{(k_{1}+1)p}\text{ }dx%
\end{array}
\label{60}
\end{equation}%
and%
\begin{equation}
\begin{array}{l}
\int_{\Omega }(\widetilde{u}^{\alpha _{2}}\widetilde{v}^{\beta _{2}}+m%
\widetilde{v}^{q-1})v_{A}^{\overline{k}_{1}q+1}\text{ }dx\leq \int_{\Omega
}u^{\alpha _{2}}v_{A}^{\overline{k}_{1}q+1+\beta _{2}}\text{ }%
dx+m\int_{\Omega }v^{(\overline{k}_{1}+1)q}\text{ }dx \\ 
\leq \int_{\Omega }u^{\alpha _{2}}v^{\overline{k}_{1}q+1+\beta _{2}}\text{ }%
dx+m\int_{\Omega }v^{(\overline{k}_{1}+1)q}\text{ }dx.%
\end{array}
\label{60*}
\end{equation}

Then, following the quite similar argument as in \cite{MM}, we obtain that $%
(u,v)\in L^{\infty }(\Omega )\times L^{\infty }(\Omega )$ and there exists a
constant $L>0,$ independent of $R$, such that $\left\Vert u\right\Vert
_{\infty },\left\Vert v\right\Vert _{\infty }\leq L$. Furthermore, from (\ref%
{30}) and (\ref{c2}), it holds%
\begin{equation}
\begin{array}{l}
\widetilde{u}^{\alpha _{1}}\widetilde{v}^{\beta _{1}}+m\widetilde{u}%
^{p-1}\leq \widetilde{u}^{\alpha _{1}}(\widetilde{v}^{\beta _{1}}+m%
\widetilde{u}^{p-1-\alpha _{1}}) \\ 
\leq \underline{u}^{\alpha _{1}}(\left\Vert v\right\Vert _{\infty }^{\beta
_{1}}+m\left\Vert u\right\Vert _{\infty }^{p-1-\alpha _{1}}) \\ 
\leq (C^{-1}\frac{c_{0}}{2}\phi _{1,p})^{\alpha _{1}}(L^{\beta
_{1}}+mL^{p-1-\alpha _{1}})\leq C_{1}d(x)^{\alpha _{1}}\text{ \ in }\Omega%
\end{array}
\label{70}
\end{equation}%
and%
\begin{equation}
\begin{array}{l}
\widetilde{u}^{\alpha _{2}}\widetilde{v}^{\beta _{2}}+m\widetilde{v}%
^{q-1}\leq \underline{v}^{\beta _{2}}(\left\Vert u\right\Vert _{\infty
}^{\alpha _{2}}+m\left\Vert v\right\Vert _{\infty }^{q-1-\beta _{2}})\leq
C_{2}d(x)^{\beta _{2}}\text{ \ in }\Omega ,%
\end{array}
\label{70*}
\end{equation}%
with positive constants $C_{1}$ and $C_{2}$. Thus, on the basis of (\ref{68}%
), (\ref{68*}), (\ref{70}), (\ref{70*}) and (\ref{500}), the nonlinear
regularity theory found in \cite{Hai} guarantees that the solutions $(u,v)$
of (\ref{50}) belong to $C^{1,\gamma }(\overline{\Omega })\times C^{1,\gamma
}(\overline{\Omega })$ for some $\gamma \in (0,1)$ and satisfy (\ref{31}).

Now, let us prove (\ref{16}). We only show the first inequality in (\ref{16}%
) because the second one can be justified similarly. To this end, we set the
functions $f,g:\Omega \rightarrow \mathbb{R}$ given by 
\begin{equation*}
f(x)=C^{-(p-1)}h_{1}(x)+\rho \max \{(\underline{u}+\varepsilon )^{\alpha
_{1}-1}R^{\beta _{1}},\underline{u}^{p-1}\}
\end{equation*}%
and 
\begin{equation*}
g(x)={f_{1,\varepsilon ,t}}({x,}\widetilde{u},\widetilde{v}).
\end{equation*}%
By Remark \ref{R3}, the strict inequalities in (\ref{18}), (\ref{29*}) and
the monotonicity of ${f_{1,\varepsilon ,t}}$ imply%
\begin{equation}
\begin{array}{l}
f(x)=-C^{-(p-1)}\phi _{1,p}^{\alpha _{1}}(x)+\rho \max \{(\underline{u}%
+\varepsilon )^{\alpha _{1}-1}R^{\beta _{1}},\underline{u}^{p-1}\} \\ 
<t(\underline{u}{+\varepsilon )}^{\alpha _{1}}\underline{v}^{\beta
_{1}}+(1-t){m}\underline{u}^{p-1}+\rho \max \{(\underline{u}{+\varepsilon )}%
^{\alpha _{1}-1}R^{\beta _{1}},\underline{u}^{p-1}\} \\ 
={f_{1,\varepsilon ,t}}({x,}\underline{u},\underline{v})\leq {%
f_{1,\varepsilon ,t}}({x,}\widetilde{u},\widetilde{v})=g(x)\text{ \ in }%
\Omega _{\delta }%
\end{array}
\label{10}
\end{equation}%
and%
\begin{equation}
\begin{array}{l}
f(x)=C^{-(p-1)}\phi _{1,p}^{\alpha _{1}}(x)+\rho \max \{(\underline{u}%
+\varepsilon )^{\alpha _{1}-1}R^{\beta _{1}},\underline{u}^{p-1}\} \\ 
<(\underline{u}{+\varepsilon )}^{\alpha _{1}}\underline{v}^{\beta _{1}}+\rho
\max \{(\underline{u}+\varepsilon )^{\alpha _{1}-1}R^{\beta _{1}},\underline{%
u}^{p-1}\}\text{ \ in }\Omega \backslash \overline{\Omega }_{\delta },%
\end{array}
\label{13**}
\end{equation}%
for all $t\in \lbrack 0,1]$ and for all $\varepsilon \in (0,1)$. On another
hand, by (\ref{c2}), (\ref{h1}), (\ref{30}), (\ref{256}) and (\ref{23}), we
obtain%
\begin{equation}
\begin{array}{l}
(\underline{u}{+\varepsilon )}^{\alpha _{1}}\underline{v}^{\beta
_{1}}=(t+1-t)(\underline{u}{+\varepsilon )}^{\alpha _{1}}\underline{v}%
^{\beta _{1}} \\ 
\leq t(\underline{u}{+\varepsilon )}^{\alpha _{1}}\underline{v}^{\beta
_{1}}+(1-t)(C^{-1}\frac{c_{0}}{2}\phi _{1,p})^{\alpha
_{1}}(C^{-1}c_{1}^{\prime }\phi _{1,q})^{\beta _{1}} \\ 
\leq t(\underline{u}{+\varepsilon )}^{\alpha _{1}}\underline{v}^{\beta
_{1}}+(1-t)(C^{-1}\frac{c_{0}}{2}\mu )^{\alpha _{1}}(C^{-1}c_{1}^{\prime
}M)^{\beta _{1}} \\ 
\leq t(\underline{u}{+\varepsilon )}^{\alpha _{1}}\underline{v}^{\beta
_{1}}+(1-t)m(C^{-1}\frac{c_{0}}{2}\mu )^{p-1} \\ 
\leq t(\underline{u}{+\varepsilon )}^{\alpha _{1}}\underline{v}^{\beta
_{1}}+(1-t){m}\underline{u}^{p-1}\text{ \ in }\Omega \backslash \overline{%
\Omega }_{\delta },%
\end{array}
\label{13***}
\end{equation}%
provided that $m>0$ sufficiently large, for all $t\in \lbrack 0,1]$ and all $%
\varepsilon \in (0,1)$. Combining (\ref{13**}) with (\ref{13***}) and using
the monotonicity of ${f_{1,\varepsilon ,t}}$, one gets%
\begin{equation}
\begin{array}{l}
f(x)=C^{-(p-1)}\phi _{1,p}^{\alpha _{1}}(x)+\rho \max \{(\underline{u}%
+\varepsilon )^{\alpha _{1}-1}R^{\beta _{1}},\underline{u}^{p-1}\} \\ 
<{f_{1,\varepsilon ,t}}({x,}\underline{u},\underline{v})\leq {%
f_{1,\varepsilon ,t}}({x,}\widetilde{u},\widetilde{v})=g(x)\text{ \ in }%
\Omega \backslash \overline{\Omega }_{\delta }%
\end{array}
\label{10*}
\end{equation}%
for all $t\in \lbrack 0,1]$ and all $\varepsilon \in (0,1)$. Consequently,
it follows from (\ref{10}) and (\ref{10*}) that for each compact set $%
K\subset \subset \Omega ,$ there is a constant $\tau =\tau (K)>0$ such that%
\begin{equation*}
\begin{array}{l}
f(x)+\tau =-C^{-(p-1)}\phi _{1,p}^{\alpha _{1}}(x)+\rho \max \{(\underline{u}%
+\varepsilon )^{\alpha _{1}-1}R^{\beta _{1}},\underline{u}^{p-1}\}+\tau \\ 
\leq {f_{1,\varepsilon ,t}}({x,}\widetilde{u},\widetilde{v})=g(x)\text{ \
a.e. in }K\cap \Omega _{\delta }%
\end{array}%
\end{equation*}%
and 
\begin{equation*}
\begin{array}{l}
f(x)+\tau =C^{-(p-1)}\phi _{1,p}^{\alpha _{1}}(x)+\rho \max \{(\underline{u}%
+\varepsilon )^{\alpha _{1}-1}R^{\beta _{1}},\underline{u}^{p-1}\}+\tau \\ 
\leq {f_{1,\varepsilon ,t}}({x,}\widetilde{u},\widetilde{v})=g(x)\text{ \
a.e. in }K\cap \Omega \backslash \overline{\Omega }_{\delta },%
\end{array}%
\end{equation*}%
for all $t\in \lbrack 0,1]$ and all $\varepsilon \in (0,1)$. Hence, given a
compact set $k\subset \subset \Omega $, there is $\tau >0$ such that 
\begin{equation*}
f(x)+\tau \leq g(x),\quad \forall x\in K
\end{equation*}%
and so, $f\prec g$ and $f,g\in L_{loc}^{\infty }(\Omega )$. Thereby, by the
strong comparison principle (see Appendix, Proposition \ref{P0}), we infer
that 
\begin{equation*}
u(x)\gg \underline{u}(x),\quad \forall x\in \Omega .
\end{equation*}%
The proof of the second inequality in (\ref{16}) is carried out in a similar
way. This complete the proof.
\end{proof}

\begin{proposition}
\label{P6}Under the assumption (\ref{h1}) problem (\ref{50}) has no
solutions for $t=0$.
\end{proposition}

\begin{proof}
Arguing by contradiction, let $(u^{\ast },v^{\ast })\in C^{1,\gamma }(%
\overline{\Omega })\times C^{1,\gamma }(\overline{\Omega }),$ for certain $%
\gamma \in (0,1)$, be a nontrivial (positive) solution of (\ref{50}) with 
\begin{equation}
(u^{\ast },v^{\ast })\in \mathcal{O}_{R}\text{ \ and }t=0.  \label{49}
\end{equation}%
From (\ref{c2}) and (\ref{30}) 
\begin{equation*}
\begin{array}{l}
\underline{u}(x)=C^{-1}z_{1}(x)\geq C^{-1}\frac{c_{0}}{2}\phi _{1,p}(x)\text{
in }\Omega \text{.}%
\end{array}%
\end{equation*}%
In the sequel, we fix $u_{1}=C^{-1}\frac{c_{0}}{2}\phi _{1,p}$ and take $%
\lambda _{\delta }=\lambda _{1,p}+\delta $ for $\delta >0$. Let $u_{2}\in
C_{0}^{1}(\overline{\Omega })$ be the solution of the problem%
\begin{equation*}
\left\{ 
\begin{array}{l}
-\Delta _{p}u_{2}=\lambda _{\delta }u_{1}^{p-1}\text{in }\Omega , \\ 
u_{2}=0\text{ on }\partial \Omega .%
\end{array}%
\right.
\end{equation*}%
Then for $\delta >0$ small and $m$ large enough, we have 
\begin{equation*}
-\Delta _{p}u_{2}=\lambda _{\delta }u_{1}^{p-1}\leq m\widetilde{u}%
^{p-1}=-\Delta _{p}u^{\ast }
\end{equation*}%
and 
\begin{equation*}
-\Delta _{p}u_{1}=\lambda _{1,p}u_{1}^{p-1}\leq \lambda _{\delta
}u_{1}^{p-1}=-\Delta _{p}u_{2}.
\end{equation*}%
By the weak comparison principle we get%
\begin{equation*}
u_{1}\leq u_{2}\leq u^{\ast }\text{ in }\Omega \text{.}
\end{equation*}%
Now let us consider the solutions of the problems%
\begin{equation*}
\left\{ 
\begin{array}{ll}
-\Delta _{p}u_{n}=\lambda _{\delta }u_{n-1}^{p-1} & \text{ in }\Omega , \\ 
u_{n}=0 & \text{ on }\partial \Omega .%
\end{array}%
\right.
\end{equation*}%
We obtain an increasing sequence $\{u_{n}\}$ such that%
\begin{equation*}
u_{1}\leq u_{n-1}\leq u_{n}\leq u^{\ast }\text{ in }\Omega \text{.}
\end{equation*}%
Passing to the limit we get a positive solution $u\in W_{0}^{1,p}(\Omega )$
for problem%
\begin{equation*}
\left\{ 
\begin{array}{ll}
-\Delta _{p}u=\lambda _{\delta }u^{p-1} & \text{ in }\Omega , \\ 
u=0 & \text{ on }\partial \Omega ,%
\end{array}%
\right.
\end{equation*}%
which is impossible for $\delta >0$ small enough because the first
eigenvalue for $p$-Laplacian is isolate. Hence, problem (\ref{50}) has no
solutions for $t=0$.
\end{proof}

Define the homotopy $\mathcal{H}_{\varepsilon }$ on $\left[ 0,1\right]
\times C^{1}(\overline{\Omega })\times C^{1}(\overline{\Omega })$ by%
\begin{equation*}
\mathcal{H}_{\varepsilon }(t,u,v)=I(u,v)-\left( 
\begin{array}{cc}
T_{p,\varepsilon }^{-1} & 0 \\ 
0 & T_{q,\varepsilon }^{-1}%
\end{array}%
\right) \times \left( 
\begin{array}{l}
{f_{1,\varepsilon ,t}}({x,}\widetilde{{u}}{,}\widetilde{{v}}) \\ 
\multicolumn{1}{c}{f{_{2,\varepsilon ,t}}({x,}\widetilde{{u}}{,}\widetilde{{v%
}})}%
\end{array}%
\right) .
\end{equation*}%
According to Lemma \ref{L1} (see Appendix) and because functions ${%
f_{\varepsilon ,t}}$ and ${g_{\varepsilon ,t}}$ belong to $C(\overline{%
\Omega })$ for all $x\in \overline{\Omega }$ and all $\varepsilon \in (0,1),$
$\mathcal{H}_{\varepsilon }$ is well defined. Furthermore, $\mathcal{H}%
_{\varepsilon }:\left[ 0,1\right] \times C^{1}(\overline{\Omega })\times
C^{1}(\overline{\Omega })\rightarrow C(\overline{\Omega })\times C(\overline{%
\Omega })$ is completely continuous for all $\varepsilon \in (0,1)$. This is
due to the compactness of the operators $T_{p,\varepsilon
}^{-1},T_{q,\varepsilon }^{-1}:C(\overline{\Omega })\rightarrow C^{1}(%
\overline{\Omega }),$ for all $\varepsilon \in (0,1)$, see appendix for more
details. Hence, $(u,v)\in \mathcal{O}_{R}$ is a solution for (\ref{pr}) if,
and only if, 
\begin{equation*}
\begin{array}{c}
(u,v)\in \mathcal{O}_{R}\,\,\,\mbox{and}\,\,\,\mathcal{H}_{\varepsilon
}(1,u,v)=0.%
\end{array}%
\end{equation*}

From the previous Proposition \ref{P2} and since $R$ is the a strict a
priori bound, it is clear that solutions of (\ref{50}) must lie in $\mathcal{%
O}_{R}$. Thus, the fact that problem (\ref{50}) has no solutions for $t=0$
(see proposition \ref{P6}) implies that 
\begin{equation*}
\deg \left( \mathcal{H}_{\varepsilon }(0,\cdot ,\cdot ),\mathcal{O}%
_{R},0\right) =0\text{\ \ for all }\varepsilon \in (0,1).
\end{equation*}%
Consequently, from the homotopy invariance property, it follows that 
\begin{equation}
\begin{array}{c}
\deg \left( \mathcal{H}_{\varepsilon }(1,\cdot ,\cdot ),\mathcal{O}%
_{R},0\right) =\deg \left( \mathcal{H}_{\varepsilon }(0,\cdot ,\cdot ),%
\mathcal{O}_{R},0\right) =0\text{ for all }\varepsilon \in (0,1).%
\end{array}
\label{35}
\end{equation}

\vspace{0.5cm}

\noindent \textbf{Topological degree: The second estimate.}

\vspace{0.5 cm}

We show that the degree of an operator corresponding to the system (\ref{pr}%
) is $1$ on the set $\mathcal{\hat{O}}$. To this end, we modify the problem
to ensure that solutions cannot occur outside of the rectangle formed by $(%
\underline{u},\underline{v})$ and $(\hat{u},\hat{v})$. Set 
\begin{equation}
\widetilde{u}=\left\{ 
\begin{array}{l}
\hat{u}\text{ if }u\geq \hat{u} \\ 
u\text{ if }\underline{u}\leq u\leq \hat{u} \\ 
\underline{u}\text{ if }u\leq \underline{u}%
\end{array}%
\right. ,\text{ \ }\widetilde{v}=\left\{ 
\begin{array}{l}
\hat{v}\text{ if }v\geq \hat{v} \\ 
v\text{ if }\underline{v}\leq v\leq \hat{v} \\ 
\underline{v}\text{ if }v\leq \underline{v},%
\end{array}%
\right.  \label{5*}
\end{equation}%
and let us define the truncation problem%
\begin{equation}
\left\{ 
\begin{array}{ll}
T_{p,\varepsilon }(u)=g_{1,\varepsilon ,t}(x,u,v) & \text{in }\Omega , \\ 
T_{q,\varepsilon }(v)=g_{2,\varepsilon ,t}(x,u,v) & \text{in }\Omega , \\ 
u,v>0\text{ in }\Omega , &  \\ 
u,v=0\text{ \ on }\partial \Omega , & 
\end{array}%
\right.  \tag{$P_{g}$}  \label{2.7}
\end{equation}%
with%
\begin{equation*}
\begin{array}{c}
g_{1,\varepsilon ,t}(x,u,v)=t(\widetilde{u}+\varepsilon )^{\alpha _{1}}%
\widetilde{v}^{\beta _{1}}+(1-t)\eta (\phi _{1,p}+\varepsilon )^{\alpha _{1}}
\\ 
+\rho \max \{(\underline{u}+\varepsilon )^{\alpha _{1}-1}R^{\beta _{1}},%
\widetilde{u}^{p-1}\},%
\end{array}%
\end{equation*}%
\begin{equation*}
\begin{array}{c}
g_{2,\varepsilon ,t}(x,u,v)=t\widetilde{u}^{\alpha _{2}}(\widetilde{v}%
+\varepsilon )^{\beta _{2}}+(1-t)\eta (\phi _{1,q}+\varepsilon )^{\beta _{2}}
\\ 
+\rho \max \{R^{\alpha _{2}}(\underline{v}+\varepsilon )^{\beta _{2}-1},%
\widetilde{v}^{q-1}\},%
\end{array}%
\end{equation*}%
with a constant $\eta >0$. The constant $\rho >0$ is chosen sufficiently
large so that the following inequalities are satisfy:{%
\begin{equation*}
\begin{array}{c}
\alpha _{1}(s_{1}{+\varepsilon )}^{\alpha _{1}-1}s_{2}^{\beta _{1}}+\rho
\max \{(\underline{u}{+\varepsilon )}^{\alpha _{1}-1}R^{\beta
_{1}},(p-1)s_{1}^{p-2}\}\geq 0,%
\end{array}%
\end{equation*}%
} uniformly in $x\in \Omega $, for $(s_{1},s_{2})\in \lbrack \underline{u},%
\hat{u}]\times \lbrack \underline{v},\hat{v}],$ for $\varepsilon \in (0,1),$
and{%
\begin{equation*}
\begin{array}{c}
\beta _{2}s_{1}^{\alpha _{2}}(s_{2}{+\varepsilon )}^{\beta _{2}-1}+\rho \max
\{R^{\alpha _{2}}(\underline{v}{+\varepsilon )}^{\beta
_{2}-1},(q-1)s_{2}^{q-2}\}\geq 0,\text{ }%
\end{array}%
\end{equation*}%
} uniformly in $x\in \Omega ,$ for $(s_{1},s_{2})\in \lbrack \underline{u},%
\hat{u}]\times \lbrack \underline{v},\hat{v}],$ for $\varepsilon \in (0,1)$.

We state the following result regarding truncation system (\ref{2.7}).

\begin{proposition}
\label{P3}Under condition (\ref{h1}) every solution $(u,v)$ of (\ref{2.7})
is in $C^{1,\gamma }(\overline{\Omega })\times C^{1,\gamma }(\overline{%
\Omega })$ for certain $\gamma \in (0,1),$ with $\left\Vert u\right\Vert
_{C^{1,\gamma }},\left\Vert v\right\Vert _{C^{1,\gamma }}<R$ and satisfies 
\begin{equation}
\begin{array}{c}
\underline{u}(x)\ll u(x)\ll \hat{u}(x)\text{ and }\underline{v}(x)\ll
v(x)\ll \hat{v}(x), \quad \forall x\in \Omega .%
\end{array}
\label{1*}
\end{equation}
\end{proposition}

\begin{proof}
A quite similar argument as in the proof of Proposition \ref{P2} provides
that all solutions of (\ref{2.7}) are in $C^{1,\gamma }(\overline{\Omega }%
)\times C^{1,\gamma }(\overline{\Omega })$ for certain $\gamma \in (0,1)$.

Let us prove (\ref{1*}). We only show the first part of inequalities in (\ref%
{1*}) because the second part can be justified similarly. To this end, we
set the functions $f,\tilde{g}:\Omega \rightarrow \mathbb{R}$ given by 
\begin{equation*}
f(x)=C^{-(p-1)}h_{1}(x)+\rho \max \{(\underline{u}+\varepsilon )^{\alpha
_{1}-1}R^{\beta _{1}},\underline{u}^{p-1}\}
\end{equation*}%
and 
\begin{equation*}
\tilde{g}(x)=g_{1,\varepsilon ,t}({x,}\widetilde{u},\widetilde{v}).
\end{equation*}%
From Remark \ref{R3}, (\ref{c2}) and (\ref{23}), for all $\varepsilon \in
(0,1)$ and for all $t\in \lbrack 0,1]$, that%
\begin{equation}
\begin{array}{l}
(t+1-t)(\underline{u}{+\varepsilon )}^{\alpha _{1}}\underline{v}^{\beta _{1}}
\\ 
\leq t(\underline{u}{+\varepsilon )}^{\alpha _{1}}\underline{v}^{\beta
_{1}}+(1-t)(C^{-1}\frac{c_{0}}{2}\phi _{1,p}+\varepsilon )^{\alpha
_{1}}(C^{-1}c_{1}^{\prime }\phi _{1,q})^{\beta _{1}} \\ 
\leq t(\underline{u}{+\varepsilon )}^{\alpha _{1}}\underline{v}^{\beta
_{1}}+(1-t)(C^{-1}\frac{c_{0}}{2}\phi _{1,p})^{\alpha
_{1}}(C^{-1}c_{1}^{\prime }M)^{\beta _{1}} \\ 
t(\underline{u}{+\varepsilon )}^{\alpha _{1}}\underline{v}^{\beta
_{1}}+(1-t)\eta (\phi _{1,p}+\varepsilon )^{\alpha _{1}}\text{ in }\Omega
\backslash \overline{\Omega }_{\delta }%
\end{array}
\label{14}
\end{equation}%
provided that $\eta >0$ is sufficiently large. Then, following the quite
similar argument which proves (\ref{16}) in Proposition \ref{P2}, we obtain
for each compact set $K\subset \Omega ,$ there is a constant $\tau =\tau
(K)>0$ such that 
\begin{equation*}
f(x)+\tau \leq \tilde{g}(x)\quad \mbox{a.e in }\quad \Omega .
\end{equation*}%
Hence, $f\prec \tilde{g}$ and $f,\tilde{g}\in L_{loc}^{\infty }(\Omega )$.
Thereby, by the strong comparison principle (see Proposition \ref{P0} in
Appendix) we infer that 
\begin{equation*}
u(x)\gg \underline{u}(x)\quad \forall x\in \Omega .
\end{equation*}
\end{proof}

Let us define the homotopy $\mathcal{N}_{\varepsilon }$ on $\left[ 0,1\right]
\times C^{1}(\overline{\Omega })\times C^{1}(\overline{\Omega })$ by%
\begin{equation}
\mathcal{N}_{\varepsilon }(t,u,v)=I(u,v)-\left( 
\begin{array}{cc}
T_{p,\varepsilon }^{-1} & 0 \\ 
0 & T_{q,\varepsilon }^{-1}%
\end{array}%
\right) \times \left( 
\begin{array}{l}
g_{1,\varepsilon ,t}({x,}u,v) \\ 
\multicolumn{1}{c}{g_{2,\varepsilon ,t}({x,}u,v)}%
\end{array}%
\right) .  \label{17}
\end{equation}%
Clearly, Lemma \ref{L1} together with Proposition \ref{P4} (see Appendix)
imply that $\mathcal{N}_{\varepsilon }$ is well defined and completely
continuous homotopy for all $\varepsilon \in (0,1)$ and all $t\in \lbrack
0,1]$. Moreover, $(u,v)\in \mathcal{\hat{O}}$ is a solution of system (\ref%
{pr}) if, and only if, 
\begin{equation*}
\begin{array}{c}
(u,v)\in \mathcal{\hat{O}}\,\,\,\mbox{and}\,\,\,\mathcal{N}_{\varepsilon
}(1,u,v)=0\text{ for all }\varepsilon \in (0,1).%
\end{array}%
\end{equation*}

In view of Proposition \ref{P3} and from the definition of function\ $\hat{u}
$ and $\hat{v}$ it follows that all solutions of (\ref{2.7}) are also
solutions of (\ref{pr}). Moreover, these solutions must be in the set $%
\mathcal{\hat{O}}$. Moreover, for $t=0$ in (\ref{17}), Minty-Browder Theorem
together with Hardy-Sobolev Inequality and \cite[Lemma 3.1]{Hai} ensure that
problems 
\begin{equation*}
\left\{ 
\begin{array}{ll}
-\Delta _{p}u=\eta (\phi _{1,p}+\varepsilon )^{\alpha _{1}} & \text{in }%
\Omega \\ 
u=0 & \text{on }\partial \Omega%
\end{array}%
\right. \text{ \ and \ }\left\{ 
\begin{array}{ll}
-\Delta _{q}v=\eta (\phi _{1,q}+\varepsilon )^{\beta _{2}} & \text{in }\Omega
\\ 
v=0 & \text{on }\partial \Omega ,%
\end{array}%
\right.
\end{equation*}%
admit unique positive solutions $\grave{u}_{\varepsilon }$ and $\grave{v}%
_{\varepsilon }$ in $C^{1,\gamma }(\overline{\Omega })$ for certain $\gamma
\in (0,1)$ and for $\varepsilon \in (0,1)$, respectively. Then, the homotopy
invariance property of the degree gives 
\begin{equation}
\begin{array}{ll}
\deg (\mathcal{N}_{\varepsilon }(1,\cdot ,\cdot ),\mathcal{\hat{O}},0) & 
=\deg (\mathcal{N}_{\varepsilon }(0,\cdot ,\cdot ),\mathcal{\hat{O}},0) \\ 
& =\deg (\mathcal{N}_{\varepsilon }(0,\cdot ,\cdot ),B_{R}(0)),0) \\ 
& =1.%
\end{array}
\label{55}
\end{equation}%
Since 
\begin{equation*}
\mathcal{H}_{\varepsilon }(1,\cdot ,\cdot )=\mathcal{N}_{\varepsilon
}(1,\cdot ,\cdot )\,\,\,\text{in}\,\,\,\mathcal{\hat{O}},
\end{equation*}%
it follows that 
\begin{equation}
\begin{array}{c}
\deg (\mathcal{H}_{\varepsilon }(1,\cdot ,\cdot ),\mathcal{\hat{O}},0)=1,%
\end{array}
\label{56}
\end{equation}%
for all $\varepsilon \in (0,1)$.

\vspace{0.5cm}

\noindent \textbf{Topological degree: The third estimate.}

\bigskip

Herafter, we will assume that 
\begin{equation*}
\mathcal{H}_{\varepsilon }(1,u,v) \not=0 \,\,\,\, \forall (u,v) \in \partial 
\mathcal{\hat{O}},
\end{equation*}
otherwise we will have a solution $(\breve{u}_{\varepsilon },\breve{v}%
_{\varepsilon }) \in \partial \mathcal{\hat{O}}$, which is different from
the solution $(u,v)$ in Theorem \ref{T1}, because $(u,v) \in \mathcal{\hat{O}%
}$. Here, we have used that $\mathcal{\hat{O}}$ is an open set, then $(u,v)
\notin \partial \mathcal{\hat{O}}$.

By (\ref{55}), (\ref{56}) and (\ref{35}), we deduce from the excision
property of Leray-Schauder degree that%
\begin{equation*}
\begin{array}{c}
\deg (\mathcal{H}_{\varepsilon }(1,\cdot ,\cdot ),\mathcal{O}_{R}\backslash 
\overline{\mathcal{\hat{O}}},0)={-1}%
\end{array}%
\end{equation*}%
and thus problem (\ref{pr}) has a solution $(\breve{u}_{\varepsilon },\breve{%
v}_{\varepsilon })\in C^{1,\gamma }(\overline{\Omega })\times C^{1,\gamma }(%
\overline{\Omega })$ for some $\gamma \in (0,1)$ with 
\begin{equation}
(\breve{u}_{\varepsilon },\breve{v}_{\varepsilon })\in \mathcal{O}%
_{R}\backslash \overline{\mathcal{\hat{O}}}  \label{61}
\end{equation}%
In view of remark (\ref{R3}), $(\breve{u}_{\varepsilon },\breve{v}%
_{\varepsilon })$ is necessarily another solution for problem (\ref{pr}).

\vspace{0.5cm}

\noindent \textbf{Proof of Theorem \ref{T2}:}

\bigskip

Set $\varepsilon =\frac{1}{n}$ with any positive integer $n\geq 1$. From (%
\ref{61}) with $\varepsilon =\frac{1}{n}$, we know that there exist $(\breve{%
u}_{n},\breve{v}_{n}):=(\breve{u}_{\frac{1}{n}},\breve{v}_{\frac{1}{n}})$
bounded in $C^{1,\gamma }(\overline{\Omega })\times C^{1,\gamma }(\overline{%
\Omega })$ for some $\gamma \in (0,1)$ such that%
\begin{equation}
\left\{ 
\begin{array}{l}
-\Delta _{p}\breve{u}_{n}=\left( \breve{u}_{n}+\frac{1}{n}\right) ^{\alpha
_{1}}\breve{v}_{n}^{\beta _{1}}\text{ in }\Omega , \\ 
-\Delta _{q}\breve{v}_{n}=\breve{u}_{n}^{\alpha _{2}}\left( \breve{v}_{n}+%
\frac{1}{n}\right) ^{\beta _{2}}\text{ in }\Omega , \\ 
\breve{u}_{n}=\breve{v}_{n}=0\text{ on }\partial \Omega ,%
\end{array}%
\right.  \label{122*}
\end{equation}%
satisfying%
\begin{equation}
\begin{array}{c}
(\breve{u}_{n},\breve{v}_{n})\in \mathcal{O}_{R}\setminus \overline{\mathcal{%
\hat{O}}}\,\,\,\forall n\in \mathbb{N}.%
\end{array}
\label{62}
\end{equation}%
Employing Arzel\`{a}-Ascoli's theorem, we may pass to the limit in $C^{1}(%
\overline{\Omega })\times C^{1}(\overline{\Omega })$ and the limit functions 
$(\breve{u},\breve{v})\in C^{1}(\overline{\Omega })\times C^{1}(\overline{%
\Omega })$ satisfy (\ref{p}) with 
\begin{equation}
(\breve{u},\breve{v})\in \mathcal{O}_{R}\setminus \overline{\mathcal{\hat{O}}%
}  \label{63}
\end{equation}%
Finally, on account of (\ref{63}) and Proposition \ref{P1}, we achieve that $%
(\breve{u},\breve{v})$ is a second solution of problem (\ref{p}). This
complete the proof of Theorem \ref{T2}.

\section{Appendix}

In this section, we establish a version of the strong comparison principle
for the operators $T_{p,\varepsilon }$ and $T_{q,\varepsilon }$ introduced
in Section \ref{S3} and we study the compactness of the inverse of these
operators. We only prove the strong comparison principle for the operator $%
T_{p,\varepsilon }$ and the compactness of $T_{p,\varepsilon }^{-1}$ because
for $T_{q,\varepsilon }$ and $T_{q,\varepsilon }^{-1}$ the proof can be
justified similarly.

\bigskip

\noindent \textbf{1. Strong comparison principle.}

\begin{proposition}
\label{P0}Let $u_{1},u_{2}\in C^{1,\beta }(\overline{\Omega }),$ $\beta \in
(0,1),$ be the solutions of the problems%
\begin{equation*}
\left\{ 
\begin{array}{ll}
T_{p,\varepsilon }(u_{1})={f(x)} & \text{in }\Omega , \\ 
u_{1}=0 & \text{on }\partial \Omega ,%
\end{array}%
\right. \text{ and\ }\left\{ 
\begin{array}{ll}
T_{p,\varepsilon }(u_{2})={g(x)} & \text{in }\Omega , \\ 
u_{2}=0 & \text{on }\partial \Omega ,%
\end{array}%
\right.
\end{equation*}%
where 
\begin{equation*}
T_{p,\varepsilon }(u)=-\Delta _{p}u+\rho \max \{(\underline{u}+\varepsilon
)^{\alpha _{1}-1}R^{\beta _{1}},\left\vert u\right\vert ^{p-2}u\},
\end{equation*}%
for some $\varepsilon \in (0,1)$ and $f,g\in L_{loc}^{\infty }(\Omega )$. If 
$f\prec g$, that is, for each compact set $K\subset \Omega $, there is $\tau
=\tau (K)>0$ such that 
\begin{equation*}
f(x)+\tau \leq g(x)\quad \mbox{a.e in}\quad K,
\end{equation*}%
then $u_{1}\ll u_{2}$.
\end{proposition}

\begin{proof}
The proof is very similar to those of Proposition 2.6 in \cite{AR}, it is
sufficient to observe that that for all $a,b,c,d\in \mathbb{R}$ the
following inequality holds: 
\begin{equation}
\begin{array}{c}
|\max \{a,b\}-\max \{c,d\}|\leq \max \{\left\vert a-c\right\vert ,\left\vert
b-d\right\vert \},%
\end{array}
\label{DES}
\end{equation}%
which leads to 
\begin{equation*}
\begin{array}{l}
|\max \{(\underline{u}+\varepsilon )^{\alpha _{1}-1}R^{\beta
_{1}},\left\vert u_{1}\right\vert ^{p-2}u_{1}\}-\max \{(\underline{u}%
+\varepsilon )^{\alpha _{1}-1}R^{\beta _{1}},\left\vert u_{2}\right\vert
^{p-2}u_{2}\}| \\ 
\leq \left\vert \left\vert u_{1}\right\vert ^{p-2}u_{1}-\left\vert
u_{2}\right\vert ^{p-2}u_{2}\right\vert .%
\end{array}%
\end{equation*}%
The last inequality is a key point in the arguments found in \cite{AR}.
\end{proof}

\noindent \textbf{2. Compactness of $T_{p,\varepsilon }$.}

Let us consider the Dirichlet problem%
\begin{equation}
\left\{ 
\begin{array}{ll}
T_{p,\varepsilon }(u)={f(x)} & \text{in }\Omega , \\ 
u=0\text{ \ } & \text{on }\partial \Omega ,%
\end{array}%
\right.  \label{1***}
\end{equation}%
where $\Omega $ is a bounded domain in $%
\mathbb{R}
^{N},$ ${f\in }W^{-1,p^{\prime }}(\Omega )$ and $T_{p,\varepsilon
}:W_{0}^{1,p}(\Omega )\rightarrow W^{-1,p^{\prime }}(\Omega )$ is the
operator defined as follows:%
\begin{equation*}
\begin{array}{c}
T_{p,\varepsilon }(u)=-\Delta _{p}u+\rho \max \{(\underline{u}+\varepsilon
)^{\alpha _{1}-1}R^{\beta _{1}},\left\vert u\right\vert ^{p-2}u\}%
\end{array}%
\end{equation*}%
for all $\varepsilon \in (0,\varepsilon _{0})$.

A solution of (\ref{1***}) is understood in the weak sense, that is $u\in
W_{0}^{1,p}(\Omega )$ satisfying%
\begin{equation}
\begin{array}{c}
\int_{\Omega }\left( |\nabla u|^{p-2}\nabla u\nabla \varphi +\rho \max \{(%
\underline{u}+\varepsilon )^{\alpha _{1}-1}R^{\beta _{1}},\left\vert
u\right\vert ^{p-2}u\}\varphi \right) \ dx=\int_{\Omega }f\left( x\right)
\varphi \ dx%
\end{array}
\label{7***}
\end{equation}%
for all $\varphi \in W_{0}^{1,p}(\Omega )$.

\begin{lemma}
\label{L1}Problem (\ref{1***}) possesses a unique solution $u_{\varepsilon }$
in $W_{0}^{1,p}(\Omega )$ for all $\varepsilon \in (0,\varepsilon _{0}).$
Moreover, if $f \in L^{\infty}(\Omega)$ the solution $u_{\varepsilon }$
belongs to $C^{1,\gamma }(\Omega ), $ for certain $\gamma \in (0,1),$ and
satisfies%
\begin{equation}
\left\Vert u_{\varepsilon }\right\Vert _{C^{1,\gamma }}<\overline{R},
\label{2***}
\end{equation}%
where $\overline{R}$ is a positive constant, which depends of $\|f\|_\infty$.
\end{lemma}

\begin{proof}
To prove the lemma we apply Minty-Browder Theorem. To do so, we prove that
the operator $T_{p,\varepsilon }$ is continuous, strict monotone and
coercive for all $\varepsilon \in (0,\varepsilon _{0})$. Let us show that $%
T_{p,\varepsilon }$ is a continuous operator. For $\{u_{n}\}\subset
W_{0}^{1,p}(\Omega )$ with $u_{n}\rightarrow u$ in $W_{0}^{1,p}(\Omega )$,
we have 
\begin{equation*}
\begin{array}{l}
\left\Vert T_{p,\varepsilon }(u_{n})-T_{p,\varepsilon }(u)\right\Vert
_{W^{-1,p^{\prime }}(\Omega )}=\underset{\varphi \in W_{0}^{1,p}(\Omega
),\left\Vert \varphi \right\Vert _{1,p}\leq 1}{\sup }\left\vert \left\langle
T_{p,\varepsilon }(u_{n})-T_{p,\varepsilon }(u),\varphi \right\rangle
\right\vert \\ 
\\ 
\leq \int_{\Omega }\left\vert \left\langle \left( |\nabla u_{n}|^{p-2}\nabla
u_{n}-|\nabla u|^{p-2}\nabla u\right) ,\nabla \varphi \right\rangle
\right\vert \text{ }dx \\ 
\\ 
+\rho \int_{\Omega }\left\vert \max \{(\underline{u}+\varepsilon )^{\alpha
_{1}-1}R^{\beta _{1}},\left\vert u_{n}\right\vert ^{p-2}u_{n}\}-\max \{(%
\underline{u}+\varepsilon )^{\alpha _{1}-1}R^{\beta _{1}},\left\vert
u\right\vert ^{p-2}u\}\right\vert \left\vert \varphi \right\vert dx.%
\end{array}%
\end{equation*}%
Then if $p\geq 2$, using \cite[Lemma $5.3$]{GM} together with H\"{o}lder's
inequality and (\ref{DES}), we derive 
\begin{equation}
\begin{array}{l}
\left\Vert T_{p,\varepsilon }(u_{n})-T_{p,\varepsilon }(u)\right\Vert
_{W^{-1,p^{\prime }}(\Omega )}\leq c_{p}\left\Vert |\nabla u|+|\nabla
u|\right\Vert _{p}^{p^{\prime }(p-2)}\left\Vert u_{n}-u\right\Vert
_{1,p}^{p^{\prime }} \\ 
\\ 
+\rho \underset{\varphi \in W_{0}^{1,p}(\Omega ),\left\Vert \varphi
\right\Vert _{1,p}\leq 1}{\sup }\int_{\Omega }\left\vert \max \{0,\left\vert
u_{n}\right\vert ^{p-2}u_{n}-\left\vert u\right\vert ^{p-2}u\}\right\vert
\left\vert \varphi \right\vert dx \\ 
\\ 
\leq C(\left\Vert u_{n}\right\Vert _{1,p}+\left\Vert u\right\Vert
_{1,p})^{p^{\prime }(p-2)}\left\Vert u_{n}-u\right\Vert _{1,p}^{p^{\prime
}}+\rho \left\Vert \left\vert u_{n}\right\vert ^{p-2}u_{n}-\left\vert
u\right\vert ^{p-2}u\right\Vert _{p^{\prime }},%
\end{array}
\label{3***}
\end{equation}%
with some constant $C>0.$ If $1<p<2$ \cite[Lemma $5.4$]{GM} and H\"{o}lder's
inequality imply that%
\begin{equation}
\begin{array}{l}
\left\Vert T_{p,\varepsilon }(u_{n})-T_{p,\varepsilon }(u)\right\Vert
_{W^{-1,p^{\prime }}(\Omega )} \\ 
\leq c_{p}\left\Vert u_{n}-u\right\Vert _{1,p}+\rho \left\Vert \left\vert
u_{n}\right\vert ^{p-2}u_{n}-\left\vert u\right\vert ^{p-2}u\right\Vert
_{p^{\prime }}.%
\end{array}
\label{6***}
\end{equation}%
Consequently, the operator $L_{p,\varepsilon }$ is continuous for all $%
\varepsilon \in (0,\varepsilon _{0})$.

Now we claim that $L_{p,\varepsilon }$\textbf{\ }is strict monotone and
coercive. Indeed, let $u_{1},u_{2}\in W_{0}^{1,p}(\Omega )$. We note that
the integral 
\begin{equation}
\begin{array}{c}
\int_{\Omega }\left( \max \{(\underline{u}+\varepsilon )^{\alpha
_{1}}R^{\beta _{1}},\left\vert u_{1}\right\vert ^{p-2}u_{1}\}-\max \{(%
\underline{u}+\varepsilon )^{\alpha _{1}}R^{\beta _{1}},\left\vert
u_{2}\right\vert ^{p-2}u_{2}\}\right) (u_{1}-u_{2})dx%
\end{array}%
\end{equation}%
is positive because 
\begin{equation}
\left( \max \{(\underline{u}+\varepsilon )^{\alpha _{1}-1}R^{\beta
_{1}},\left\vert u_{1}\right\vert ^{p-2}u_{1}\}-\max \{(\underline{u}%
+\varepsilon )^{\alpha _{1}-1}R^{\beta _{1}},\left\vert u_{2}\right\vert
^{p-2}u_{2}\}\right) (u_{1}-u_{2})\geq 0\text{ \ in }\Omega .
\end{equation}%
Then for all $\varepsilon \in (0,\varepsilon _{0})$ we have%
\begin{equation*}
\begin{array}{l}
\left\langle T_{p,\varepsilon }(u_{1})-T_{p,\varepsilon
}(u_{2}),u_{1}-u_{2}\right\rangle =\int_{\Omega }\left\langle \left( |\nabla
u_{1}|^{p-2}\nabla u_{1}-|\nabla u_{2}|^{p-2}\nabla u_{2}\right) ,\nabla
(u_{1}-u_{2})\right\rangle \text{ }dx \\ 
\\ 
+\rho \int_{\Omega }\left( \max \{(\underline{u}+\varepsilon )^{\alpha
_{1}-1}R^{\beta _{1}},\left\vert u_{1}\right\vert ^{p-2}u_{1}\}-\max \{(%
\underline{u}+\varepsilon )^{\alpha _{1}-1}R^{\beta _{1}},\left\vert
u_{2}\right\vert ^{p-2}u_{2}\}\right) (u_{1}-u_{2})dx \\ 
\\ 
\geq \int_{\Omega }\left\langle \left( |\nabla u_{1}|^{p-2}\nabla
u_{1}-|\nabla u_{2}|^{p-2}\nabla u_{2}\right) ,\nabla
(u_{1}-u_{2})\right\rangle \text{ }dx%
\end{array}%
\end{equation*}%
and the claim follows due to the strict monotonicity of $-\Delta _{p}$ in $%
W_{0}^{1,p}(\Omega )$. The coercivity of the operator $T_{1,\varepsilon }$
can be proved easily using the coercivity of $-\Delta _{p}$. Now we are able
to apply the Minty-Browder theorem which guarantees the existence of a
unique solution for problem (\ref{1***}) in $W_{0}^{1,p}(\Omega )$.

Next we show that solutions $u_{\varepsilon }$ of (\ref{1***}) are in $%
C^{1,\gamma }(\overline{\Omega }),$ for certain $\gamma \in (0,1)$ for all $%
\varepsilon \in (0,\varepsilon _{0})$. The proof is based on Moser's
iterations technique combined with nonlinear regularity theory (see \cite{L}%
).

For $M>0$, define on $\Omega $ the function $u_{\varepsilon ,M}\left(
x\right) =\min \left( u_{\varepsilon }\left( x\right) ,M\right) .$ We act on
(\ref{7***}) with $\varphi =u_{\varepsilon ,M}^{k_{1}p+1}$ where 
\begin{equation}
\begin{array}{c}
\left( k_{1}+1\right) p=p^{\ast }%
\end{array}
\label{50***}
\end{equation}%
which gives 
\begin{equation}
\begin{array}{l}
\int_{\Omega }\left( \left( k_{1}p+1\right) \left\vert \nabla u_{\varepsilon
,M}\right\vert ^{p}u_{\varepsilon ,M}^{k_{1}p}+\rho \max \{(\underline{u}%
+\varepsilon )^{\alpha _{1}-1}R^{\beta _{1}},\left\vert u_{\varepsilon
}\right\vert ^{p-2}u_{\varepsilon }\}u_{\varepsilon ,M}^{k_{1}p+1}\right) 
\text{ }dx \\ 
=\int_{\Omega }f(x)u_{\varepsilon ,M}^{k_{1}p+1}\text{ }dx%
\end{array}
\label{36***}
\end{equation}%
By the Sobolev embedding theorem, the left-hand side of (\ref{36***}) is
estimated from below as follows 
\begin{equation}
\begin{array}{l}
\int_{\Omega }\left( \left( k_{1}p+1\right) \left\vert \nabla u_{\varepsilon
,M}\right\vert ^{p}u_{\varepsilon ,M}^{k_{1}p}+\rho \max \{(\underline{u}%
+\varepsilon )^{\alpha _{1}-1}R^{\beta _{1}},\left\vert u_{\varepsilon
}\right\vert ^{p-2}u_{\varepsilon }\}u_{\varepsilon ,M}^{k_{1}p+1}\right) 
\text{ }dx \\ 
\geq \int_{\Omega }\left( (k_{1}p+1)|\nabla u_{\varepsilon
,M}|^{p}u_{\varepsilon ,M}^{k_{1}p}+\rho \left\vert u_{\varepsilon
}\right\vert ^{p-2}u_{\varepsilon }\text{ }u_{\varepsilon
,M}^{k_{1}p+1}\right) \\ 
\geq \int_{\Omega }\left( (k_{1}p+1)|\nabla u_{\varepsilon
,M}|^{p}u_{\varepsilon ,M}^{k_{1}p}+\rho u_{\varepsilon
,M}^{(k_{1}+1)p}\right) \\ 
=\frac{k_{1}p+1}{\left( k_{1}+1\right) ^{p}}\int_{\Omega }\left\vert \nabla
u_{\varepsilon ,M}^{k_{1}+1}\right\vert ^{p}+\rho \left\Vert u_{\varepsilon
,M}\right\Vert _{p^{\ast }}^{p^{\ast }}\geq C_{1}\frac{\left(
k_{1}p+1\right) }{\left( k_{1}+1\right) ^{p}}\left\Vert u_{\varepsilon
,M}\right\Vert _{(k_{1}+1)p^{\ast }}^{p^{\ast }}%
\end{array}
\label{l6***}
\end{equation}%
where $C_{1}$ is some positive constant. From (\ref{50***})$,$ the
right-hand side of (\ref{36***}) is estimated from above by%
\begin{equation}
\begin{array}{l}
\int_{\Omega }f(x)u_{\varepsilon ,M}^{k_{1}p+1}\leq \left\Vert f\right\Vert
_{\infty }\int_{\Omega }u_{\varepsilon }^{k_{1}p+1}\leq \left\Vert
f\right\Vert _{\infty }\left\Vert u_{\varepsilon }\right\Vert _{p^{\ast
}}^{k_{1}p+1}.%
\end{array}
\label{8***}
\end{equation}%
Following the same arguments as in \cite{MM} we obtain that $u_{\varepsilon
}\in L^{\infty }(\Omega )$ for all $\varepsilon \in (0,\varepsilon _{0})$.
Then from the nonlinear regularity theory (see \cite{L}) we infer that $%
u_{\varepsilon }\in C^{1,\gamma }(\overline{\Omega }),$ for certain $\gamma
\in (0,1)$ and $\left\Vert u_{\varepsilon }\right\Vert _{C^{1,\gamma }}<%
\overline{R}$ for a large constant $\overline{R}>0$ and for all $\varepsilon
\in (0,\varepsilon _{0})$.
\end{proof}

Lemma \ref{L1} ensures that the inverse operator 
\begin{equation*}
T_{p,\varepsilon }^{-1}:C(\overline{\Omega })\rightarrow C^{1}(\overline{%
\Omega })
\end{equation*}%
is well defined for all $\varepsilon \in (0,\varepsilon _{0})$. The next
proposition gives some properties regarding $T_{p,\varepsilon }^{-1}.$

\begin{proposition}
\label{P4}The operator $T_{p,\varepsilon }^{-1}$ is continuous and compact
for all $\varepsilon \in (0,\varepsilon _{0})$.
\end{proposition}

\begin{proof}
First, let us show that $T_{p,\varepsilon }^{-1}$ is a continuous operator.
So let $f_{n}\rightarrow f$ in $C(\overline{\Omega })$. Denoting $%
u_{n}=T_{p,\varepsilon }^{-1}(f_{n})$ reads as%
\begin{equation}
\begin{array}{c}
\int_{\Omega }\left( |\nabla u_{n}|^{p-2}\nabla u_{n}\nabla \varphi +\rho
\max \{(\underline{u}+\varepsilon )^{\alpha _{1}-1}R^{\beta _{1}},\left\vert
u_{n}\right\vert ^{p-2}u_{n}\}\varphi \right) \ dx=\int_{\Omega }f_{n}\left(
x\right) \varphi \ dx%
\end{array}
\label{4***}
\end{equation}%
for all $\varphi \in W_{0}^{1,p}(\Omega )$. Since by (\ref{2***}) the
sequence $\{u_{n}\}$ is bounded in $W_{0}^{1,p}(\Omega )$, along a relabeled
subsequence there holds 
\begin{equation}
u_{n}\rightharpoonup u\text{ with some }u\in W_{0}^{1,p}(\Omega ).
\label{5***}
\end{equation}%
Setting $\varphi =u_{n}-u$ in (\ref{4***}). Then Lebesgue's dominated
convergence theorem ensures%
\begin{equation*}
\begin{array}{c}
\underset{n\rightarrow \infty }{\lim }\left\langle -\Delta
_{p}u_{n},u_{n}-u\right\rangle =0.%
\end{array}%
\end{equation*}%
The $S_{+}$ property of $-\Delta _{p}$ on $W_{0}^{1,p}(\Omega )$ along with (%
\ref{5***}) implies $u_{n}\rightarrow u$ in $W_{0}^{1,p}(\Omega )$.
Furthermore, the boundedness of the sequence $\{u_{n}\}$ in $C^{1,\gamma }(%
\overline{\Omega })$ and since the embedding $C^{1,\gamma }(\overline{\Omega 
})\subset C^{1}(\overline{\Omega })$ is compact, it turns out that along a
relabeled subsequence, one has the fact that $u_{n}\rightarrow u$ in $C^{1}(%
\overline{\Omega })$. Finally, (\ref{4***}) result in $u=T_{p,\varepsilon
}^{-1}(f)$, proving that $T_{p,\varepsilon }^{-1}$ is continuous operator.

Next, we show that $T_{p,\varepsilon }^{-1}(C(\overline{\Omega }))$ is a
relatively compact subset of $C^{1}(\overline{\Omega })$. Let $%
u_{n}=T_{p,\varepsilon }^{-1}(f_{n})$ with $f_{n}\in C(\overline{\Omega })$
for all $n$. Following the same reasoning as before, we find $u\in C^{1}(%
\overline{\Omega })$ such that, along a relabeled subsequence, $%
u_{n}\rightarrow u$ in $C^{1}(\overline{\Omega })$, thereby the relative
compactness of $T_{p,\varepsilon }^{-1}$ is proven.
\end{proof}

\begin{acknowledgement}
The work was accomplished while the second author was visiting the
University Federal of Campina Grande with CNPq-Brazil fellowship N$%
{{}^\circ}%
$ 402792/2015-7. He thanks for hospitality.
\end{acknowledgement}


\begin{thebibliography}{99}
\bibitem{AC} C. O. Alves \& F. J. S. A. Corr\^{e}a, \emph{On the existence
of positive solution for a class of singular systems involving quasilinear
operators}, Appl. Math. Comput. 185 (2007), 727-736.

\bibitem{ACG} C. O. Alves, F. J. S. A. Corr\^{e}a \& J.V.A. Gon\c{c}alves, 
\emph{Existence of solutions for some classes of singular Hamiltonian
systems,} Adv. Nonlinear Stud. 5 (2005), 265-278.

\bibitem{CM} C. O. Alves \& A. Moussaoui, \emph{Existence of solutions for a
class of singular elliptic systems with convection term}, Asymptotic
Analysis 90 (2014), 237-248.

\bibitem{AR} D. Arcoya \& D. Ruiz, \emph{The Ambrosetti-Prodi problem for
the p-Laplace operator}, Comm. Partial Diff. Eqts. 31 (2006), 849-865.

\bibitem{AO} R.P. Agarwal \& D. O'Regan, \emph{Existence theory for single
and multiple solutions to singular positive boundary value problems,} J.
Diff. Equat. 175 (2001), 393-414.

\bibitem{B} H. Br\'{e}zis, \emph{Analyse fonctionnelle theorie et
applications}, Masson, Paris, 1983.

\bibitem{CLM} S. Carl, V. K. Le \& D. Motreanu, \emph{Nonsmooth variational
problems and their inequalities}. Comparaison principles and applications,
Springer, New York, 2007.

\bibitem{CFMT} P. Cl\'{e}ment, J. Fleckinger, E. Mitidieri \& F. De Thelin, 
\emph{Existence of Positive Solutions for a Nonvariational Quasilinear
Elliptic System},\textit{\ }J. Diff. Eqts. 166 (2000), 455-477.

\bibitem{CP} M. M. Coclite \& G. Palmieri, \emph{On a singular nonlinear
Dirichlet problem,} Comm. Partial Diff. Equat. 14 (1989), 1315-1327.

\bibitem{CR} M. G. Crandall, P. H. Rabinowitz \& L. Tartar,\emph{\ On a
Dirichlet problem with singular nonlinearity,} Comm. Partial Diff. Equat. 2
(1977), 193-222.

\bibitem{D} M. del Pino, \emph{A priori estimates applications to
existence-nonexistence for a semilinear elliptic system,} Ind. University
Math. J. 43 (1994), 77-129.

\bibitem{DMO} I. Diaz, J. M. Morel \& L. Oswald, \emph{An elliptic equation
with singular nonlinearity,} Comm. Partial Diff. Equat. 12 (1987), 1333-1344.

\bibitem{FM} W. Fulks \& J. S. Maybee, \emph{A singular non-linear equation,}
Osaka Math. J. 12 (1960), 1-19.

\bibitem{G} M. Ghergu, \emph{Lane-Emden systems with negative exponents}, J.
Funct. Anal. 258 (2010), 3295-3318.

\bibitem{GHM} J. Giacomoni, J. Hernandez \& A. Moussaoui, \emph{Quasilinear
and singular systems: the cooperative case}, Contemporary Math. 540 (2011),
Amer. Math. Soc., Providence, R.I., 79-94.

\bibitem{GHS} J. Giacomoni, J. Hernandez \& P. Sauvy, \emph{Quasilinear and
singular elliptic systems}, Advances Nonl. Anal. 2 (2013), 1-41.

\bibitem{GST} J. Giacomoni, I. Schindler \& P. Takac, \emph{Sobolev versus H%
\"{o}lder local minimizers and existence of multiple solutions for a
singular quasilinear equation}, A. Sc. N. Sup. Pisa (5) 6 (2007), 117-158.

\bibitem{GM} R. Glowinski \& A. Marroco, \emph{Sur l'approximation par \'{e}l%
\'{e}ments finis d'ordre un, et la r\'{e}solution, par p\'{e}%
nalisation-dualit\'{e} d'une classe de probl\`{e}mes de Dirichlet non lin%
\'{e}aires}, Univ. Paris VI et CNRS, 189, nr. 74023.

\bibitem{GR} M. Ghergu \& V. Radulescu, \emph{On a class of Gierer-Meinhardt
systems arising in morphogenesis}, \textit{C. R. Acad. Sci. Paris}, \textit{%
Ser.I} 344\textbf{\ }(2007), 163-168.

\bibitem{GM1} A. Gierer \& H. Meinhardt, \emph{A theory of biological
pattern formation,} Kybernetik 12 (1972), 30-39.

\bibitem{Hai} D. D. Hai, \emph{On a class of singular p-Laplacian boundary
value problems}, J. Math. Anal. Appl. 383 (2011), 619-626.

\bibitem{HMV} J. Hernandez, F. J. Mancebo \& J. M. Vega, \emph{Positive
solutions for singular semilinear elliptic systems}, Adv. Diff. Eqts. 13
(2008), 857-880.

\bibitem{K} O. Kavian, \emph{Introduction \`{a} la Th\'{e}orie des Points
Critiques et Applications aux Probl\`{e}mes Elliptiques}, Springer,
Paris-Berlin-Heidelberg, 1993.

\bibitem{KM} B. Khodja \& A. Moussaoui, \emph{Positive solutions for
infinite semipositone}$/$\emph{positone quasilinear elliptic systems with
singular and superlinear terms}, Submitted.

\bibitem{LM} A. C. Lazer \& P. J. Mckenna, \emph{On a singular nonlinear
elliptic boundary-value problem}, Proc. American Math. Soc. 3 (111), 1991.

\bibitem{L} G. M. Lieberman, \emph{\ Boundary regularity for solutions of
degenerate elliptic equations}, Nonlinear Anal. 12 (1988), 1203-1219.

\bibitem{LP} C. D. Luning \& W. L. Perry, \emph{Positive solutions of
negative exponent generalized Emden-Fowler boundary value problem,} SIAM J.
Math. Anal. 12 (1981), 874-879.

\bibitem{MM2} D. Motreanu \& A. Moussaoui, \emph{A quasilinear singular
elliptic system without cooperative structure, }Act. Math. Sci. 34 B (3)
(2014), 905-916.

\bibitem{MM3} D. Motreanu \& A. Moussaoui, \emph{An existence result for a
class of quasilinear singular competitive elliptic systems}, Applied Math.
Letters 38 (2014), 33-37.

\bibitem{MM} D. Motreanu \& A. Moussaoui, \emph{Existence and boundedness of
solutions for a singular cooperative quasilinear elliptic system}, Complex
Var. Elliptic Eqts. 59 (2014), 285-296.

\bibitem{MKT} A. Moussaoui, B. Khodja \& S. Tas, \emph{A singular
Gierer-Meinhardt system of elliptic equations in }$R^{N}$, Nonlinear Anal.
71 (2009), 708-716.

\bibitem{ST} C. A. Stuart, \emph{Existence and approximations of solutions
of nonlinear elliptic equations,} Math. Z. 147 (1976), 53-63.

\bibitem{T} S. Taliaferro, \emph{A nonlinear singular boundary value problem,%
} Nonlinear Anal. Theory Methods Appl. (1979) 897-904.

\bibitem{V} J. L. Vazquez, \emph{A strong maximum principle for some
quasilinear elliptic equations}. Appl. Math. Optim. 12 (1984),191-202.

\bibitem{Z1} Z. Zhang, \emph{On a Dirichlet with a singular nonlinearity,}
J. Math. Anal. Appl. 194 (1995), 103-113.

\bibitem{ZY} Z. Zhang \& J. Yu, \emph{On a singular nonlinear Dirichlet
problem with a convection term,} SIAM J. Math. Anal. 32 (2000), 916-927.
\end{thebibliography}
\end{document}